\documentclass[12pt,diff_eng]{article} 
\usepackage[english]{babel}
\usepackage{graphicx}
\usepackage{amsmath}
\usepackage{amssymb}
\usepackage{amsthm}
\usepackage{textcomp}
\usepackage{hyperref}
\usepackage{chngcntr}
\usepackage{apptools}
\usepackage{verbatim}
\usepackage[T2A]{fontenc}
\usepackage[symbol]{footmisc}
\usepackage{footnote}

\usepackage{diff_eng}
\AtAppendix{\counterwithin{theorem}{section}}
\AtAppendix{\counterwithin{proposition}{section}}

\newtheorem{theorem}{Theorem}

\newtheorem{lemma}{Lemma}

\theoremstyle{definition}
\newtheorem{problem}{Problem}

\theoremstyle{remark}
\newtheorem{remark}{Remark}
\newtheorem{corollary}{Corollary}

\numberwithin{equation}{section}
\numberwithin{lemma}{section}

\newcommand\restr[2]{{
		\left.\kern-\nulldelimiterspace 
		#1 
		\vphantom{\big|} 
		\right|_{#2} 
}}

\begin{document}
  \thispagestyle{empty}

\parbox{0.275\textwidth }{
\begin{picture}(130,120)
\put(67,67){{\Huge $\frac{dx}{dt}$}}
\put(40,0){\vector(0,1){70}}
\put(40,70){\line(0,1){10}}
\qbezier(40,80)(40,110)(80,110)
\qbezier(80,110)(120,110)(120,80)
\put(120,80){\line(0,-1){10}}
\qbezier(120,70)(120,35)(75,35)
\put(75,35){\vector(-1,0){5}}
\put(0,35){\line(1,0){70}}
\put(0,40){\vector(1,0){10}}
\qbezier(10,40)(35,40)(35,70)
\put(35,70){\line(0,1){10}}
\qbezier(35,80)(35,115)(80,115)
\qbezier(80,115)(125,115)(125,80)
\put(125,80){\line(0,-1){10}}
\qbezier(125,70)(125,30)(75,30)
\put(75,30){\line(-1,0){15}}
\qbezier(60,30)(45,30)(45,15)
\put(45,15){\vector(0,-1){15}}
\end{picture}
}
\hfill
\noindent \parbox{0.45\textwidth }{\footnotesize \it
\begin{center}
DIFFERENTIAL EQUATIONS\\ 
AND\\
CONTROL PROCESSES

\noindent N. ?, 2022\\

\noindent Electronic Journal, \\
reg. N ${\Phi}$C77-39410 at 15.04.2010

\noindent ISSN 1817-2172
\medskip

\noindent http://diffjournal.spbu.ru/\\

\noindent e-mail: jodiff@mail.ru

\end{center}}
\vspace*{10mm}
\large

\rightline{\underline{\it Delay differential equations}}

\bigskip

  \date{}
  
  \bigskip
  \noindent\parbox{\linewidth}{\centering\LARGE{\textbf{Nonlinear semigroups for delay equations in Hilbert spaces, inertial manifolds and dimension estimates} }}
  
  \bigskip
  
  \noindent\parbox{\linewidth}{
  	\centering
  	{\large Mikhail Anikushin\footnote[1]{demolishka@yandex.ru.}}
  	\\ \vspace{0.2cm} \normalsize\text{Department of
  		Applied Cybernetics, Faculty of Mathematics and Mechanics,}
  	\\ \vspace{0.2cm} \normalsize\text{Saint Petersburg State University.}
  }
  
  \bigskip
  
  \noindent\textbf{Abstract.}
We study the well-posedness of nonautonomous nonlinear delay equations in $\mathbb{R}^{n}$ as evolutionary equations in a proper Hilbert space. We present a construction of solving operators (nonautonomous case) or nonlinear semigroups (autonomous case) for a large class of such equations. The main idea can be easily extended for certain PDEs with delay. Our approach has lesser limitations and much more elementary than some previously known constructions of such semigroups and solving operators based on the theory of accretive operators. In the autonomous case we also study differentiability properties of these semigroups in order to apply various dimension estimates using the Hilbert space geometry. However, obtaining effective dimension estimates for delay equations is a nontrivial problem and we explain it by means of a scalar delay equation. We also discuss our adjacent results concerned with inertial manifolds and their construction for delay equations.
  \bigskip
  
  \noindent\textbf{Keywords:} Delay equations, Nonlinear semigroups, Inertial manifolds, Dimension estimates.
  \newpage
  
\section{Introduction}

In this paper we consider the following class of nonlinear nonautonomous delay differential equations in $\mathbb{R}^{n}$:
\begin{equation}
\label{EQ: ClassicalDelayEquation}
\dot{x}(t) = \widetilde{A}x_{t} + \widetilde{B}F(t,\widetilde{C}x_{t}),
\end{equation}
where $x_{t}(\theta) := x(t+\theta)$, $\theta \in [-\tau,0]$, denotes the history segment; $\tau>0$ is a constant; $\widetilde{A} \colon C([-\tau,0];\mathbb{R}^{n}) \to \mathbb{R}^{n}$, $\widetilde{B} \colon \mathbb{R}^{m} \to \mathbb{R}^{n}$ and $\widetilde{C} \colon C([-\tau,0];\mathbb{R}^{n}) \to \mathbb{R}^{r}$ are bounded linear operators and $F \colon \mathbb{R} \times \mathbb{R}^{r} \to \mathbb{R}^{m}$ is a nonlinear continuous function such that for some constant $\Lambda=\Lambda(t) > 0$, which is bounded in $t$ from compact intervals, we have
\begin{equation}
\label{EQ: DelayLipschitz}
|F(t,y_{1}) - F(t,y_{2})| \leq \Lambda(t) |y_{1}-y_{2}| \text{ for all } y_{1},y_{2} \in \mathbb{R}^{r}, t \in \mathbb{R}.
\end{equation}
Here and below by $|\cdot|$ we denote the Euclidean norm in $\mathbb{R}^{j}$ for any $j>0$. For $t \geq s$ we put $\Lambda_{s}^{t} := \sup_{s \leq \theta \leq t}\Lambda(\theta)$.

From the classical theory (that is the application of the Banach fixed point theorem), it follows that for any $\phi_{0} \in C([-\tau,0];\mathbb{R}^{n})$ and $t_{0} \in \mathbb{R}$ there exists a unique classical solution $x(\cdot)=x(\cdot,t_{0},\phi_{0}) \colon [t_{0}-\tau,+\infty) \to \mathbb{R}^{n}$, i.~e. such that $x_{t_{0}} \equiv \phi_{0}$, $x(\cdot) \in C^{1}([t_{0},+\infty);\mathbb{R}^{n})$ and $x(\cdot)$ satisfies \eqref{EQ: ClassicalDelayEquation} for $t \geq t_{0}$. We define the family of solving operators $\widetilde{U}(t,s) \colon C([-\tau,0];\mathbb{R}^{n}) \to C([-\tau,0];\mathbb{R}^{n})$, where $t \geq s$, by $\widetilde{U}(t,s)\phi_{0}:=x_{t}(\cdot,s,\phi_{0})$, where $x_{t}(\theta,s,\phi_{0}) = x(t+\theta,s,\phi_{0})$ for $\theta \in [-\tau,0]$.

Consider the Hilbert space $\mathbb{H} := \mathbb{R}^{n} \times L_{2}(-\tau,0;\mathbb{R}^{n})$ with the usual product norm, which we denote by $|\cdot|_{\mathbb{H}}$, i.~e. for $(x,\phi) \in \mathbb{H}$ we have
\begin{equation}
	|(x,\phi)|^{2}_{\mathbb{H}} = |x|^{2} + \int_{-\tau}^{0}|\phi(\theta)|^{2}d\theta.
\end{equation}
Consider the operator $A \colon \mathcal{D}(A) \subset \mathbb{H} \to \mathbb{H}$ given by
\begin{equation}
\label{EQ: OperatorADefinition}
(x, \phi)
\overset{A}{\mapsto}
\left(\widetilde{A}\phi, \frac{d}{d \theta} \phi\right),
\end{equation}
where $(x,\phi) \in \mathcal{D}(A) := \{ (x,\phi) \in \mathbb{H} \ | \ \phi(0)=x, \phi \in W^{1,2}(-\tau,0;\mathbb{R}^{n})  \}$. From the monograph of A.~B\'{a}tkai and S.~Piazzera \cite{BatkaiPiazzera2005} we have that $A$ is a generator of a $C_{0}$-semigroup in $\mathbb{H}$ (see Lemma \ref{LEM: DelayODEGeneratorLemma} below). The bounded linear operator $B \colon \mathbb{R}^{m} \to \mathbb{H}$ is defined as $B \xi := (\widetilde{B}\xi,0)$ and we define the unbounded linear operator $C \colon \mathbb{H} \to \mathbb{R}^{r}$ as $C(x,\phi):=\widetilde{C}\phi$ for $\phi \in C([-\tau,0];\mathbb{R}^{n})$. Now \eqref{EQ: ClassicalDelayEquation} can be written as an abstract evolution equation in $\mathbb{H}$:
\begin{equation}
\label{EQ: AbstractDelayHilberSpace}
\dot{v}(t) = Av(t) + BF(t,Cv(t)),
\end{equation}
for which we will study the question of well-posedness. In our Theorem \ref{TH: DelaySemigroupTh} below we precisely state in what sense solutions to this evolution equation can be understood.

We put $\mathbb{E}:=C([-\tau,0];\mathbb{R}^{n})$ and consider the embedding $\mathbb{E} \subset \mathbb{H}$ given by $\phi \mapsto (\phi(0),\phi)$. We identify the elements of $\mathbb{E}$ and $\mathbb{H}$ under such an embedding. For any $T>0$ let $\mathcal{H}_{T}$ be the set of all continuous functions $v \colon [0,T] \to \mathbb{H}$ for which there exists a continuous function $x \colon [-\tau,T] \to \mathbb{R}^{n}$ such that $v(t) = (x(t),x_{t})$ for all $t \geq 0$. We will also make the use of the following property, which is, in fact, the main ingredient of our proofs in the first part.
\begin{description}
	\item[\textbf{(MES)}] There is a constant $M_{C}>0$ such that for all $T>0$ the inequality
	\begin{equation}
	\int_{0}^{T}|Cv(t)| dt \leq M_{C} \left( | v(0) |_{\mathbb{H}} + \|v(\cdot)\|_{L_{1}(0,T;\mathbb{H})} \right).
	\end{equation}
	is satisfied for all $v(\cdot) \in \mathcal{H}_{T}$.
\end{description}
To explain it, consider, for example, the case of $n=1$ and $r=1$. Let $v(t)=(x(t),x_{t})$ and $\widetilde{C}\phi:=\phi(-\tau)$. Then we have (for convenience, let $T \geq \tau$)
\begin{equation}
\begin{split}
\int_{0}^{T}|Cv(t)|dt = \int_{0}^{T}|x(t-\tau)|dt = \int_{-\tau}^{0}|x(t)|dt + \int_{0}^{T-\tau}|x(t)|dt \leq\\
\leq 
(1+\sqrt{\tau})\left( |v(0)|_{\mathbb{H}} + \|v(\cdot)\|_{L_{1}(0,T;\mathbb{H})} \right).
\end{split}
\end{equation}
In fact, we can use the density of $\delta$-functionals in the weak-* topology of $C([-\tau,0];\mathbb{R})$ and approximations by such functionals given by the Riesz representation theorem to show that \textbf{(MES)} holds for any operator $\widetilde{C}$ (with similar arguments in the case $n \geq 1$, $r \geq 1$, see Lemma \ref{LEM: DelayMeasuEstimate} below). This simple observation besides the results contained in the present paper also led to a version of the Frequency Theorem for delay equations \cite{Anikushin2020Freq}. The $L_{2}$-version of \textbf{(MES)} helps to prove differentiability properties of constructed semigroups (see Section \ref{SEC: DelayDifferentiability}).

One of our main results is the following.
\begin{theorem}
	\label{TH: DelaySemigroupTh}
	Suppose for \eqref{EQ: ClassicalDelayEquation} that $F$ satisfies \eqref{EQ: DelayLipschitz}. Then for any $t_{0} \in \mathbb{R}$ and $v_{0} \in \mathbb{H}$ there is a unique generalized solution $v(t)=v(t,t_{0},v_{0})$ to \eqref{EQ: AbstractDelayHilberSpace}, which is a continuous function $[t_{0},+\infty) \to \mathbb{H}$. This solution is uniquely determined by the property that the family of solving operators $U(t,s)v_{0}:=v(t,s,v_{0})$, $t \geq s$, in $\mathbb{H}$ agrees with the family $\widetilde{U}(t,s)$ on $\mathbb{E}$. Moreover, the following properties are satisfied
	\begin{description}
		\item[\textbf{(ULIP)}] For any $T>0$ and $s \in \mathbb{R}$ there is a constant $M_{1}=M_{1}(T,\Lambda_{s}^{s+T})>0$ such that for all $t \in [s,s+T]$ and $v_{1},v_{2} \in \mathbb{H}$ we have
		\begin{equation}
		|U(t,s)v_{1} - U(t,s)v_{2}|_{\mathbb{H}} \leq M_{1} |v_{1}-v_{2}|_{\mathbb{H}}.
		\end{equation}
	    Moreover, for all $t \geq s + \tau$ we have $U(t,s)\mathbb{H} \subset \mathbb{E}$ and for any $T \geq \tau$, $s \in \mathbb{R}$ and $t \in [s+\tau,s+T]$ we have
		\begin{equation}
		\|U(t,s)v_{1} - U(t,s)v_{2}\|_{\mathbb{E}} \leq M_{1} |v_{1}-v_{2}|_{\mathbb{H}}.
		\end{equation}
		\item[\textbf{(COM)}] The map $U(t,s) \colon \mathbb{E} \to \mathbb{E}$ is compact for $t \geq s + \tau$ and the map $U(t,s) \colon \mathbb{H} \to \mathbb{H}$ is compact for $t \geq s+ 2 \tau$.
		\item[(\textbf{REG})] For $v_{0} \in \mathcal{D}(A)$ the generalized solution $v(t)=v(t,t_{0},v_{0})$, $t \geq t_{0}$, is a classical solution to \eqref{EQ: AbstractDelayHilberSpace}, i.~e. we have $v(\cdot) \in C^{1}([t_{0},+\infty);\mathbb{H}) \cap C([t_{0},+\infty);\mathbb{E})$, $v(t) \in \mathcal{D}(A)$ and $v(t)$ satisfies \eqref{EQ: AbstractDelayHilberSpace} for all $t \geq t_{0}$.
		\item[\textbf{(VAR)}] For all $v_{0} \in \mathbb{E}$ and $t_{0} \in \mathbb{R}$ we have the variation of constants formula satisfied for $v(t)=v(t;t_{0},v_{0})$ and $t \geq t_{0}$ as
		\begin{equation}
			\label{EQ: VOCformulaDelay}
			v(t)=G(t-t_{0})v_{0}+\int_{t_{0}}^{t}G(t-s)BF(s,Cv(s))ds,
		\end{equation}
	    where $G(t)$ is the $C_{0}$-semigroup generated by $A$ in $\mathbb{H}$.
	\end{description}
\end{theorem}
\begin{remark}
	One can interpret the variation of constants formula from \eqref{EQ: VOCformulaDelay} for any $v_{0} \in \mathbb{H}$. See Remark \ref{REM: VariationOfConstantsDelay} for this.
\end{remark}
\noindent The proof of Theorem \ref{TH: DelaySemigroupTh} is given in the next section. Note that $L_{p}$-versions of \textbf{(MES)} (see Lemma \ref{LEM: DelayMeasuEstimate} below) can be used to show the well-posedness in $\mathbb{R}^{n} \times L_{p}(-\tau,0;\mathbb{R}^{n})$.
\begin{remark}
	\label{REM: UniformConstant}
	For our further investigations it is important that if the constant $\Lambda(t)$ from \eqref{EQ: DelayLipschitz} can be chosen independently of $t$, then the constants from \textbf{(ULIP)} depend only on $T$ (or $t-s$), but not on $s$.
\end{remark}

Our approach for Theorem \ref{TH: DelaySemigroupTh} is very simple and it is based on three steps. The first one is the well-posedness of the linear problem in $\mathbb{H}$, including exponential estimates and the variation of constants formula. Such results for (partial) delay equations are given in \cite{BatkaiPiazzera2005}. The second step is the existence of classical solutions (in the space of continuous functions $\mathbb{E}$) for the nonlinear problem, which give rise to classical solutions of the non-linear problem in $\mathbb{H}$. The existence of solutions in $\mathbb{E}$ is well-known for delay equations in $\mathbb{R}^{n}$ (see, for example, J.K.~Hale \cite{Hale1977}) as well as for some partial delay equations (see, for example, I.~Chueshov \cite{Chueshov2015}). The third step is the derivation of elementary a priori estimates for the norm of classical solutions in $\mathbb{H}$ with the use of the variation of constants formula, estimates for the linear problem, \eqref{EQ: DelayLipschitz} and \textbf{(MES)}. This allows us to obtain generalized solutions by continuity. Thus, the conclusion of Theorem \ref{TH: DelaySemigroupTh} can be easily extended to some partial differential equations with delay.

The well-posedness of nonlinear autonomous and nonautonomous (partial) delay equations in Banach spaces was studied in several papers, for example, \cite{Faheem1987, Webb1976,Webb1981}. The main approach in these papers is based on applications of the theory of accretive operators. Besides the fact that the theory itself is nonelementary and its applications require pages and pages of various estimates, in applications to delay equations one faces with several restrictions. For example, to apply results of G.~F.~Webb \cite{Webb1976} and G.F.~Webb and M.~Badii \cite{Webb1981} to ODEs in $\mathbb{R}^{n}$ with delay, the number of ``independent'' discrete delays, roughly speaking, cannot exceed $n$ that is unnatural. In \cite{Webb1981} there is also assumed some smoothness of the right-hand side in $t$, which is linked with the construction of a family of equivalent norms to obtain the accretiveness condition. In the paper of M.~Faheem and M.R.M.~Rao \cite{Faheem1987} only the case of nonautonomous delay differential equations in $\mathbb{R}^{n}$ is considered. Their main restriction is posed on the nonlinear part that must be everywhere defined in $\mathbb{H}$ (i.~e. no discrete delays can appear in the nonlinear term) and some smoothness in $t$ is also assumed. However, the authors of \cite{Faheem1987} considered a nonautonomous linear part and showed the well-posedness for the linear problem, but their assumptions on the linear part in our context allow to consider it as a nonlinear part of \eqref{EQ: ClassicalDelayEquation}, i.~e. \eqref{EQ: DelayLipschitz} and \textbf{(MES)} are satisfied. Thus, our Theorem \ref{TH: DelaySemigroupTh} covers in most and largely extends the final result of \cite{Faheem1987} as well as results from \cite{Webb1976,Webb1981}. Moreover, if we put $\widetilde{A} = 0$ in \eqref{EQ: ClassicalDelayEquation}, then the generation of a $C_{0}$-semigroup for the operator $A$ is easier to obtain. Now, if we make $F(t,\widetilde{C}x_{t})$ to be a linear function of $x_{t}$ for each $t$, then Theorem \ref{TH: DelaySemigroupTh} can be considered as a well-posedness theorem for linear nonautonomous delay equations in $\mathbb{H}$. There are other papers on the question of well-posedness for delay equations, which use more concrete approaches (see D.~Breda \cite{Breda2010} and links therein), but none of them entirely covers the result of Theorem \ref{TH: DelaySemigroupTh} and cannot compete with the simplicity of its proof.

We also note that the studying of delay equations in the Hilbert space setting is rarely seen in works on dynamical systems. For example, recent monographs on infinite-dimensional dynamical systems (see A.N.~Carvalho, J.A.~Langa and J.C.~Robinson \cite{CarvalhoLangaRobinson2012}; I.~Chueshov \cite{Chueshov2015}) treat delay equations in the space of continuous functions. Below we will justify advantages of the Hilbert space geometry for understanding of the dynamics of such equations.

Namely, the obtained variation of constants formula \textbf{(VAR)} is useful for applications of the Frequency Theorem \cite{Anikushin2020Freq,Likhtarnikov1977} (a theorem that allows to construct quadratic Lyapunov functionals, which we use to construct inertial manifolds) since it allows to pass from the nonlinear equation to a linear inhomogeneous system, which is studied in the context of the theorem\footnote{See \cite{Anikushin2020FreqParab,Anikushin2020Freq} for discussions on developments of the Frequency Theorem starting from the first infinite-dimensional version proved by V.A.~Yakubovich and A.L.~Likhtarnikov \cite{Likhtarnikov1977}.}. It is also convenient for studying of almost periodic in time equations, where, as it is well-known, one should compactify the equation by considering the so-called limiting equations and check the continuous dependence with respect to perturbations of the right-hand side. We refer to our works \cite{Anikushin2020Geom,Anikushin2021AADyn} for applications to such and other equations in the infinite-dimensional context. In finite-dimensions the Frequency Theorem (also known as the Kalman-Yakubovich-Popov lemma) has already proved to be useful for studying of dimensional-like properties of autonomous and almost periodic ODEs (see, for example, M.M.~Anikushin \cite{Anikushin2019SmithRed, Anikushin2019Vestnik}; M.M.~Anikushin, V.~Reitmann and A.O.~Romanov \cite{AnikushinRR2019}; N.V.~Kuznetsov and V.~Reitmann \cite{KuzReit2020} and references therein for a range of applications).

Theorem \ref{TH: DelaySemigroupTh} is also convenient for studying of differentiability properties of semigroups (see Section \ref{SEC: DelayDifferentiability}) and inertial manifolds (see Section \ref{SEC: InvariantManifoldsDelay}). We do not explicitly state these results in this introduction because strict formulations require long preparations. Along with Theorem \ref{TH: DelaySemigroupTh} these results form a basis for parts of our adjacent works \cite{Anikushin2020Geom,Anikushin2021AADyn,Anikushin2021SS} concerned with delay equations.

A part of the present work is motivated by our previous paper \cite{Anikushin2020Red} on the existence of invariant topological manifolds for cocycles in Hilbert spaces and especially Problem 2 posed therein, which asks for extensions of the theory for delay equations with discrete delays. At Appendix \ref{APP: AdditionToReduction} we present such an extension of the theory, which also gives a solution to Problem 1 from \cite{Anikushin2020Red} asking for conditions to provide continuous dependence of the inertial manifold fibres. In fact, under some natural additional assumptions these manifolds possess classical properties such as exponential tracking and normal hyperbolicity, which are proved and discussed within a more general context in our adjacent work \cite{Anikushin2020Geom}. In the present paper we discuss these extensions for delay equations in Section \ref{SEC: InvariantManifoldsDelay}.

Concerning inertial manifolds for delay equations, here we should mention the paper of C.~Chicone \cite{Chicone2003}, which extends classical results of Yu.A.~Ryabov and R.D.~Driver, devoted to inertial manifolds for delay equations. The theory considered in \cite{Chicone2003} allows to construct $n$-dimensional inertial manifolds for delay equations in $\mathbb{R}^{n}$ with small delays. For comparison, our approach has some delicateness in applications, where the Frequency Theorem for delay equations \cite{Anikushin2020Freq} is used. In particular, it allows to construct inertial manifolds with the dimension that is not limited to $n$ as in \cite{Chicone2003}. Moreover, it is shown in \cite{Anikushin2020Freq} that results from \cite{Chicone2003} can be deduced (with relaxed in some cases conditions) in most from the Frequency Theorem and our general theory from \cite{Anikushin2020Geom} (which, as we have mentioned, uses results of the present paper) after simple computations.

Our results from \cite{Anikushin2020Geom,Anikushin2020FreqParab,Anikushin2020Freq} allow to consider various inertial manifolds theories within a general geometric context based on the use of quadratic Lyapunov functionals. In particular, it is shown by the author in \cite{Anikushin2020FreqParab} that the Spectral Gap Condition, used by C.~Foias, G.R.~Sell and R.~Temam \cite{Temam1997} (see also the survey of S.~Zelik \cite{Zelik2014}) is a particular case of some frequency inequality arising in various versions of the Frequency Theorem, which provides the existence of quadratic functionals. This frequency inequality was also used by R.A.~Smith \cite{Smith1992,Smith1994} for his developments of the Poincar\'{e}-Bendixson theory for delay equations in $\mathbb{R}^{n}$ and certain reaction-diffusion equations\footnote{Note that R.A.~Smith did not use quadratic Lyapunov functionals in the infinite-dimensional case, probably due to his inability to show their existence. Our extensions of the Frequency Theorem presented in \cite{Anikushin2020Freq,Anikushin2020FreqParab} show that under Smith's frequency inequality such functionals do exist.}. Moreover, in the surveys of A.~Kostianko et al. \cite{KostiankoZelikSA2020} or S.~Zelik \cite{Zelik2014} it is shown that the Spatial Averaging Principle, which was proposed by J.~Mallet-Paret and G.~R.~Sell to relax the Spectral Gap Condition in some cases, also lead to the existence of certain quadratic functionals. For more discussions in this direction see \cite{Anikushin2020Geom,Anikushin2020FreqParab,Anikushin2020Freq}. Note also that our geometric approach can be also applied to construct stable/unstable and local center manifolds along invariant sets. 

Besides the above mentioned construction of quadratic Lyapunov functionals, an advantage may be given by possible applications of dimension estimates, which use the Hilbert space geometry, and approximation of dimension-like characteristics and spectra (see D.~Breda and E.~Van Vleck \cite{Breda2014}). In this direction we prove in Section \ref{SEC: DelayDifferentiability} the $C^{1}$-differentiability property for semigroups in $\mathbb{H}$ generated by \eqref{EQ: ClassicalDelayEquation} with $F \in C^{1}$ and $F'$ globally bounded. This allows to apply well-known dimension estimates for the fractal dimension (see V.V.~Chepyzhov and A.A.~Ilyin \cite{ChepyzhovIlyin2004}), the Hausdorff dimension (see R.~Temam \cite{Temam1997}, J.~Mallet-Paret \cite{MalletParet1976}) and the topological entropy (see N.V.~Kuznetsov and V.~Reitmann \cite{KuzReit2020}) of compact invariant sets. However, obtaining effective dimension estimates for delay equations seems to be a nontrivial problem and we try to discuss its nontriviality in Section \ref{SEC: DelayDifferentiability}. This can be seen also from rare papers on the topic, where the pioneering paper of J.~Mallet-Paret \cite{MalletParet1976} and a paper of J.W-H.~So and J.~Wu \cite{SoWu1991} should be mentioned. Both of the papers represent an approach, which somehow utilizes compactness of the differentials to show finite-dimensionality (as well as in the mentioned monographs \cite{CarvalhoLangaRobinson2012,Chueshov2015,Hale1977}) and does not provide any effective estimates. It seems that the authors of \cite{SoWu1991} were the first to mention some nontriviality of the problem, but they did not provide any discussions on it (referring to subsequent works that never appeared). Moreover, in \cite{MalletParretNussbaum2013} J.~Mallet-Paret and R.D.~Nussbaum studied compound processes arising from linear nonhomogeneous scalar delay equations with monotone feedback (such equations may appear after linearization). Their nontrivial results show that such compound processes for certain powers (corresponding either to odd or even-dimensional volumes growth) preserves a convex reproducing normal cone of rank 1. This allows us to treat the problem in the general context of spectral theory. Here a result of D.~Dragi\v{c}evi\'{c} \cite{Dragicevic2018}, which describes endpoints of the Sacker-Sell spectral intervals as Lyapunov exponents over some ergodic measures, are crucial for understanding of the problem. In Section \ref{SEC: ExampleDelay}, by means of the Suarez-Schopf model for ENSO \cite{Suarez1988}, we use this approach to obtain effective conditions for the nonexistence of periodic orbits and homoclinics. To the best of our knowledge, this is the first time when the criterion of R.A.~Smith \cite{Smith1986HD} and its developments by M.Y.~Li and J.S.~Muldowney \cite{LiMuldowney1995} are effectively applied for equations with delay.

This paper is organized as follows. In Section \ref{SEC: ProofSemigroupTh} we prove Theorem \ref{TH: DelaySemigroupTh}. In Section \ref{SEC: DelayDifferentiability} we prove a theorem on $C^{1}$-differentiability of the semigroup given by \eqref{EQ: AbstractDelayHilberSpace} in the autonomous case (Theorems \ref{LEM: DelayDifferentiabilityLemma} and \ref{TH: DifferentiabilityTh}). Here we also consider the problem of obtaining effective dimension estimates (see Problem \ref{PROB: ProblemEffectiveDimEstDelay}) and discuss its nontriviality. In Section \ref{SEC: ExampleDelay} we consider the Suarez-Schopf model for El Ni\~{n}o given by a scalar delay equation, pose a problem linked with dimension estimates and give a partial solution. In Section \ref{SEC: InvariantManifoldsDelay} we discuss inertial manifolds and their properties, including $C^{1}$-differentiability (see Theorem \ref{TH: SmoothnessManifoldDelay}), normal hyperbolicity and exponential tracking. Moreover, we obtain dimension estimates using some Riemannian metric, which naturally arises from the construction of inertial manifolds (see Theorem \ref{TH: FrequencyDimEstimateDelay}). In Section \ref{SEC: ExampleDelayIM} we continue our investigation of the Suarez-Schopf model and obtain conditions for the existence of one-dimensional and two-dimensional inertial manifolds (see Theorems \ref{TH: SSmodelOneDimIMSharper} and \ref{TH: SSmodelTwoDimIMSharper}). At Appendix \ref{APP: AdditionToReduction} we present a generalization of our main result from \cite{Anikushin2020Red}, which is, in particular, convenient for delay equations.
\section{Construction of delay semigroups}
\label{SEC: ProofSemigroupTh}

At first we state here a lemma from \cite{Anikushin2020Freq}, which in particular shows that \textbf{(MES)} is satisfied for any operator $\widetilde{C}$. Let $T>0$ and consider the subspace $\mathcal{S}_{T} \subset C([0,T];C([-\tau, 0];\mathbb{R}^{n})$ of all continuous functions $\phi \colon [0,T] \to C([-\tau,0];\mathbb{R}^{n})$ such that there exists a continuous function $x \colon [-\tau, T] \to \mathbb{R}^{n}$ with the property $\phi(t) = x_{t}$ for all $t \geq 0$. The following lemma is Lemma 8 from \cite{Anikushin2020Freq}.
\begin{lemma}
	\label{LEM: DelayMeasuEstimate}
	Let $\widetilde{C} \colon C([-\tau,0];\mathbb{R}^{n}) \to \mathbb{R}^{r}$ be a bounded linear operator. Then there exists a constant $M=M(\widetilde{C})>0$ such that for all $T>0$, $p \geq 1$ and any $\phi \in \mathcal{S}_{T}$ we have
	\begin{equation}
	\label{EQ: FuncLemmaLPIneq}
	\left(\int_{0}^{T} |\widetilde{C}\phi(t)|^{p} dt\right)^{1/p} \leq M_{\widetilde{C}} \cdot \left( \| \phi(0) \|^{p}_{L_{p}(-\tau,0;\mathbb{R}^{n})} + \| x(\cdot) \|^{p}_{L_{p}(0,T;\mathbb{R}^{n})} \right)^{1/p}.
	\end{equation}
\end{lemma}
Using Lemma \ref{LEM: DelayMeasuEstimate} with $p=1$, we get that \textbf{(MES)} with $M_{C}:= M_{\widetilde{C}} (1+\sqrt{\tau})$ is satisfied for the operator $C$ corresponding to $\widetilde{C}$ from \eqref{EQ: ClassicalDelayEquation}.

\begin{remark}
	From the Riesz representation theorem, there exists a $r \times n$-matrix-valued function of bounded variation $\gamma(\theta)$, where $\theta \in [-\tau,0]$, such that
	\begin{equation}
		\widetilde{C}\phi = \int_{-\tau}^{0}d\gamma(\theta) \phi(\theta) \text{ for any } \phi \in C([-\tau,0];\mathbb{R}^{n}).
	\end{equation}
    It can be shown that $M_{\widetilde{C}} \leq \operatorname{Var}\gamma$, where $\operatorname{Var}$ is the total variation on $[-\tau,0]$ (see \cite{Anikushin2020Freq}).
\end{remark}
\begin{remark}
	\label{REM: VariationOfConstantsDelay}
		Let us consider the space $\mathcal{H}^{p}_{T}$ of all functions $\phi(\cdot) \in C([0,T];L_{p}(-\tau,0);\mathbb{R}^{n})$ such that for some $x(\cdot) \in L_{p}(-\tau,0;\mathbb{R}^{n})$ we have $\phi(t)=x_{t}$ for all $t \geq 0$. It is convenient to endow $\mathcal{H}^{p}_{T}$ with the norm
		\begin{equation}
			\| \phi(\cdot) \|_{\mathcal{H}^{p}_{T}}:= \left( \|\phi(0)\|^{p}_{L_{p}(-\tau,0;\mathbb{R}^{n})} + \| x(\cdot) \|^{p}_{L_{p}(0,T;\mathbb{R}^{n})}  \right)^{1/p}.
		\end{equation}
		Now Lemma \ref{LEM: DelayMeasuEstimate} shows that the operator $\mathcal{H}^{p}_{T} \ni \phi(\cdot) \mapsto \mathcal{I}_{C}(\phi(\cdot)) \in L_{p}(0,T;\mathbb{R}^{r})$, where
		\begin{equation}
			\mathcal{I}_{C}(\phi(\cdot))(s) = \widetilde{C}\phi(s) \text{ for } \phi \in \mathcal{S}_{T} \text{ and } s \in [0,T]
		\end{equation}
		is well-defined. Moreover, its norm is independent of $T$ and $p$. This is the sense in which the variation of constants formula from \eqref{EQ: VOCformulaDelay} can be understood for any $v_{0} \in \mathbb{H}$.
\end{remark}

The following lemma is Theorem 3.23 from \cite{BatkaiPiazzera2005}.
\begin{lemma}
	\label{LEM: DelayODEGeneratorLemma}
	The operator $A$ given by \eqref{EQ: OperatorADefinition} is the generator of a $C_{0}$-semigroup $G(t)$, $t \geq 0$, in $\mathbb{H}$. In particular, there are constants $M_{A}, \varkappa_{0}>0$ such that
	\begin{equation}
	\label{EQ: ExponentialEstimate}
	|G(t)v_{0}|_{\mathbb{H}} \leq M_{A} e^{\varkappa_{0}t}|v_{0}|_{\mathbb{H}} \text{ for all } t \geq 0 \text{ and } v_{0} \in \mathbb{H}.
	\end{equation}
\end{lemma}

The following lemma is a particular case of Lemma 3.6 from \cite{BatkaiPiazzera2005}.
\begin{lemma}
	\label{LEM: ClassicalSolutionHistoryFunction}
	Let $x \colon [-\tau,\infty) \to \mathbb{R}^{n}$ be a function such that $x \in W^{1,2}_{loc}([-\tau,\infty);\mathbb{R}^{n})$. Define the history function\footnote{Here, as in the previous section, $x_{t}(\theta):=x(t+\theta)$ for $\theta \in [-\tau,0]$.} $\phi(t):=x_{t}$ for $t \geq 0$. Then $\phi \in C^{1}([0,+\infty); L_{2}(-\tau,0;\mathbb{R}^{n}))$ and for all $t \geq 0$ we have $\phi(t) \in W^{1,2}(-\tau,0;\mathbb{R}^{n})$ and
	\begin{equation}
	\dot{\phi}(t) = \frac{d}{d\theta} \phi(t) \text{ in } L_{2}(-\tau,0;\mathbb{R}^{n}).
	\end{equation}
\end{lemma}

An immediate consequence of Lemma \ref{LEM: ClassicalSolutionHistoryFunction} is the following.
\begin{corollary}
	\label{COR: CorollaryClassicalSolutions}
	Let $x(t)=x(t,t_{0},\phi_{0})$ be the classical solution to \eqref{EQ: ClassicalDelayEquation} such that $\phi_{0} \in W^{1,2}(-\tau,0;\mathbb{R}^{n})$. Then the function $v(t)=(x(t),x_{t})$, $t \geq t_{0}$, is a classical solution to \eqref{EQ: AbstractDelayHilberSpace} with $v_{0}:=v(t_{0})=\phi_{0}$ (or, more precisely, $v(t_{0})=(\phi_{0}(0),\phi_{0})$). Moreover, if we put $\xi(t):=F(Cv(t))$, then $v(\cdot)$ also satisfies the inhomogeneous equation
	\begin{equation}
	\dot{v}(t) = Av(t) + B\xi(t)
	\end{equation}
	and, in particular, the variation of constants formula for $t \geq t_{0}$
	\begin{equation}
	\label{EQ: VariationOfConstants}
	v(t) = G(t-t_{0})v_{0} + \int_{t_{0}}^{t} G(t-s)BF(Cv(s))ds
	\end{equation}
	is valid.
\end{corollary}

\begin{remark}
	\label{REM: RemarkLinearPartLemma}
	An analog of Lemma \ref{LEM: DelayODEGeneratorLemma} for parabolic or hyperbolic equations can also be proved (see, for example, Theorem 3.29 in \cite{BatkaiPiazzera2005}). To proceed from Lemma \ref{LEM: ClassicalSolutionHistoryFunction} to Corollary \ref{COR: CorollaryClassicalSolutions} we have to establish the well-posedness (=existence of classical solutions) in the space of continuous functions. This is also well-studied for partial differential equations with delay. See, for example, \cite{Chueshov2015}.
\end{remark}

Now we can give a proof of Theorem \ref{TH: DelaySemigroupTh}.

\begin{proof}
	Let $x_{j}(t)=x_{j}(t,t_{0},\phi_{0,j})$, where $t \geq t_{0}$ and $j=1,2$, be two classical solutions to \eqref{EQ: ClassicalDelayEquation} and put $v_{j}(t) :=(x(t),x_{t})$. If we suppose that $\phi_{0,j} \in W^{1,2}(-\tau,0;\mathbb{R}^{n})$, then $v_{j}(\cdot)$ satisfies \eqref{EQ: VariationOfConstants} due to Corollary \ref{COR: CorollaryClassicalSolutions}. From this, \eqref{EQ: ExponentialEstimate}, \eqref{EQ: DelayLipschitz} and \textbf{(MES)} we get
	\begin{equation}
	\begin{split}
	|v_{1}(t)-v_{2}(t)|_{\mathbb{H}} \leq\\
	\leq M_{A}e^{\varkappa_{0}(t-t_{0})}|v_{1}(t_{0})-v_{2}(t_{0})|_{\mathbb{H}} + M_{A}\Lambda_{t_{0}}^{t_{0}+T} \|B\| \int_{t_{0}}^{t}e^{\varkappa_{0}(t-s)} |C(v_{1}(s)-v_{2}(s))| ds \leq \\
	\leq (M_{A} + M_{A}M_{C}\Lambda_{t_{0}}^{t_{0}+T} \|B\| ) e^{\varkappa_{0}(t-t_{0})}|v_{1}(t_{0})-v_{2}(t_{0})|_{\mathbb{H}} +\\+ M_{A}M_{C}\Lambda_{t_{0}}^{t_{0}+T} \|B\|e^{\varkappa_{0}(t-t_{0})}\int_{t_{0}}^{t}|v_{1}(s)-v_{2}(s)|_{\mathbb{H}}ds.
	\end{split}
	\end{equation}
    From this, by applying the Gronwall lemma, for some constant $M_{1} = M_{1}(T,\Lambda_{t_{0}}^{t_{0}+T})$ we get
	\begin{equation}
	\label{EQ: EstimateLIP}
	|v_{1}(t)-v_{2}(t)|_{\mathbb{H}} \leq M_{1} |v_{1}(t_{0})-v_{2}(t_{0})|_{\mathbb{H}} \text{ for all } t \in [t_{0},t_{0}+T].
	\end{equation}
	From \eqref{EQ: EstimateLIP} for any $t_{0} \in \mathbb{R}$, $v_{0} \in \mathbb{H}$ we can define a generalized solution $v(t,t_{0},v_{0})$, $t \geq t_{0}$, by the continuity and density of $W^{1,2}(-\tau,0;\mathbb{R}^{n})$ in $\mathbb{E}$ and $\mathbb{H}$. Indeed, let a sequence $v_{0,k} \in W^{1,2}(-\tau,0;\mathbb{R}^{n})$, where $k=1,2,\ldots$, tend to $v_{0}$ in $\mathbb{H}$ as $k \to +\infty$. Then \eqref{EQ: EstimateLIP} shows that the sequence $v_{k}(t)=v_{k}(t,t_{0},v_{0,k})$, $t \in [t_{0},t_{0}+T]$, is fundamental in $\mathbb{H}$ for any $T > 0$. Its limit is the generalized solution $v(t,t_{0},v_{0})$, $t \geq t_{0}$, which is independent on the choice of $v_{0,k}$. This proves the initial statement of Theorem \ref{TH: DelaySemigroupTh} and \textbf{(ULIP)}. Moreover for $T \geq t \geq t_{0}+\tau$ we have
	\begin{equation}
	\|v_{1}(t)-v_{2}(t)\|_{\mathbb{E}} \leq \sup_{\theta \in [-\tau,0]}|v_{1}(t+\theta)-v_{2}(t+\theta)|_{\mathbb{H}} \leq M_{1}|v_{1}(t_{0})-v_{2}(t_{0})|_{\mathbb{H}}.
	\end{equation}
	This proves that, in fact, the sequence $v_{k}(t)$, $t \in [t_{0},t_{0}+T]$, defined above is fundamental in $\mathbb{E}$ for $t \geq t_{0}+\tau$ and, consequently, $v(t,t_{0},v_{0}) \in \mathbb{E}$. This proves the smoothing property stated in \textbf{(ULIP)}. In particular, the map $U(t,s)$ for $t \geq s + \tau$ takes bounded sets in $\mathbb{H}$ into bounded sets in $\mathbb{E}$, where $U(t,s)$ coincides with $\widetilde{U}(t,s)$. From the Arzel\'{a}-Ascoli theorem, the map $\widetilde{U}(t,s)$ for $t \geq s + \tau$ takes bounded sets in $\mathbb{E}$ into precompact sets in $\mathbb{E}$. Consequently, $U(t,s) \colon \mathbb{H} \to \mathbb{E}$ is compact for $t \geq s+2\tau$. This shows \textbf{(COM)}. Summarizing the above, it is clear that \textbf{(REG)} is also satisfied. From Corollary \ref{COR: CorollaryClassicalSolutions} and continuity arguments we get \textbf{(VAR)}. The proof is finished.
\end{proof}
\section{Differentiability of delay semigroups}
\label{SEC: DelayDifferentiability}
In this section we suppose that \eqref{EQ: ClassicalDelayEquation} is autonomous, i.~e. $F$ is independent of $t$. We suppose also that $F \in C^{1}(\mathbb{R}^{r}; \mathbb{R}^{m})$. Note that from \eqref{EQ: DelayLipschitz} it immediately follows that $F'$ is bounded. Let the assumptions of Theorem \ref{TH: DelaySemigroupTh} hold. Then there is a semiflow $\varphi^{t} \colon \mathbb{H} \to \mathbb{H}$ given by equation \eqref{EQ: AbstractDelayHilberSpace}, i.~e. it is defined as $\varphi^{t}(v_{0}):=v(t,0,v_{0})$ for all $t \geq 0$ and $v_{0} \in \mathbb{H}$. For $v_{0} \in \mathbb{E}$ one can formally write from \eqref{EQ: AbstractDelayHilberSpace} the linearized along the trajectory $\varphi^{t}(v_{0})$ equation
\begin{equation}
\label{EQ: LinearizedEquation}
\dot{V}(t) = [A + BF'(C\varphi^{t}(v_{0})) C ] V(t)
\end{equation}

Putting $\widetilde{F}(t,y):=F'(C\varphi^{t}(u_{0})) y$, we see that \eqref{EQ: LinearizedEquation} is of the form \eqref{EQ: AbstractDelayHilberSpace} with $F$ changed to $\widetilde{F}$. Thus, we have the following well-posedness result for \eqref{EQ: LinearizedEquation}, which follows\footnote{Although formally Theorem \ref{TH: DelaySemigroupTh} is proved for nonlinearities $F(t,y)$, which are defined for $t \in \mathbb{R}$, it clearly remains true (with obvious modifications) when we have $t \in [0,+\infty)$ only.} from Theorem \ref{TH: DelaySemigroupTh}.
\begin{lemma}
	\label{LEM: LinearizatinoWellPosed}
	For any $v_{0} \in \mathbb{H}$ and $\xi_{0} \in \mathbb{H}$ equation \eqref{EQ: LinearizedEquation} has a unique generalized solution $V(t)=V(t;\xi_{0};v_{0})$ such that $V(0)=\xi_{0}$. For $\xi_{0} \in \mathcal{D}(A)$ the solution $V(t)=V(t;\xi_{0};v_{0})$ is a classical solution, i.~e. $V(\cdot) \in C^{1}([0,+\infty);\mathbb{H}) \cap C([0,+\infty);\mathbb{E})$, $V(t) \in \mathcal{D}(A)$ and $V(t)$ satisfies \eqref{EQ: LinearizedEquation} for all $t \geq 0$.
\end{lemma}
There is, in fact, continuous dependence of solutions to \eqref{EQ: LinearizedEquation} on $v_{0} \in \mathbb{E}$ that can be seen from the variation of constants formula \textbf{(VAR)} (see Lemma \ref{LEM: DelayLinearCocycle}).

Let $\mathcal{K} \subset \mathbb{H}$ be an invariant compact set, i.~e. $\varphi^{t}(\mathcal{K}) = \mathcal{K}$ for all $t \geq 0$. From the smoothing property in \textbf{(ULIP)} and the fact that a compact invariant set consists of complete trajectories we have the inclusion $\mathcal{K} \subset \mathcal{D}(A)$. Let us formulate the following property.
\begin{description}
	\item[\textbf{(MES*)}] There is a constant $M^{*}_{C}>0$ such that for all $T>0$
	\begin{equation}
	\int_{0}^{T}|Cv(t)|^{2} dt \leq M^{*}_{C} \left( | v(0) |^{2}_{\mathbb{H}} + \|v(\cdot)\|^{2}_{L_{2}(0,T;\mathbb{H})} \right).
	\end{equation}
	for all $v(\cdot) \in \mathcal{H}_{T}$. 
\end{description}
From Lemma \ref{LEM: DelayMeasuEstimate} with $p=2$ we get that \textbf{(MES*)} with $M^{*}_{C}:= (M_{\widetilde{C}})^{2}$ is satisfied for any $C$ corresponding to $\widetilde{C}$ from \eqref{EQ: ClassicalDelayEquation}.

For the proof of differentiability properties below we need a simple observation given by the following lemma.
\begin{lemma}
	\label{LEM: DelayEquationsBoundedOperatorContinuity}
	Let $Q \colon \mathbb{R}^{r} \to \mathbb{R}^{m}$ be a globally bounded continuous function and fix some $T>0$. Then the operator $\mu(\cdot) \mapsto Q(\mu(\cdot))$ is a continuous map from $L_{2}(0,T;\mathbb{R}^{r})$ to $L_{2}(0,T;\mathbb{R}^{m})$
\end{lemma}
\begin{proof}
	One can prove the continuity at every point by contradiction, combining the facts that any convergent $L_{2}$ sequence contains a subsequence, which converges almost everywhere, and the Dominated Convergence Theorem.
\end{proof}

\begin{theorem}
	\label{LEM: DelayDifferentiabilityLemma}
	Suppose that $F \in C^{1}(\mathbb{R}^{r};\mathbb{R}^{m})$ and $F'$ is globally bounded. Then for any $v_{0} \in \mathbb{H}$, any $T>0$ and any bounded in $\mathbb{H}$ subset $\mathcal{B}$ we have
	\begin{equation}
	\lim\limits_{h \to 0}\frac{|\varphi^{t}(v_{0}+h\xi)-\varphi^{t}(v_{0}) - hV(t;v_{0};\xi)|_{\mathbb{H}}}{h} \to 0,
	\end{equation}
	where the limit is uniform in $t \in [0,T]$ and $\xi \in \mathcal{B}$.
\end{theorem}
\begin{proof}
	Put $\delta(t) = \varphi^{t}(v_{0}+h\xi)-\varphi^{t}(v_{0})-V(t;v_{0};h\xi)$. We apply the variation of constants formulas (see Remark \ref{REM: VariationOfConstantsDelay}) to get the representations
	\begin{equation}
	\varphi^{t}(v_{0}+h\xi) = G(t)(v_{0}+h\xi) + \int_{0}^{t}G(t-s)BF(C\varphi^{s}(v_{0}+h\xi))ds
	\end{equation}
	and for $V(t)=V(t;v_{0};h\xi)$
	\begin{equation}
	V(t)=hG(t)\xi + \int_{0}^{t}G(t-s)BF'(C\varphi^{s}(v_{0}))CV(s)ds.
	\end{equation}
    Putting $\Delta(s):=\varphi^{s}(v_{0}+h\xi)-\varphi^{s}(v_{0})$ and using the Newton-Liebniz formula, we get
    \begin{equation}
    	\begin{split}
    		F(C\varphi^{s}(v_{0}+h\xi))-F(C\varphi^{s}v_{0}) =\\= \left( \int_{0}^{1} F'\left[C(\alpha\varphi^{s}(v_{0}+h\xi) + (1-\alpha) \varphi^{s}(v_{0}))\right] d\alpha \right) C\Delta(s).
    	\end{split}
    \end{equation}
	Putting $\mu(s) = \mu(s;v_{0},\alpha,h,\xi) := C(\alpha\varphi^{s}(v_{0}+h\xi) + (1-\alpha) \varphi^{s}(v_{0}))$ and $\mu_{0}(s)=\mu_{0}(s;v_{0}) = C\varphi^{s}(v_{0})$, we may write $\delta(t)$ as
	\begin{equation}
	\begin{split}
	\delta(t) = \int_{0}^{t}G(t-s)B \left[\int_{0}^{1}F'(\mu(s)) d\alpha C\Delta(s) - F'(C\varphi^{s}(v_{0}))CV(s)\right]ds.
	\end{split}
	\end{equation}
	Since $V(s)=\Delta(s)-\delta(s)$, we have
	\begin{equation}
	\label{EQ: DelayDiffDeltaEstimate}
	\begin{split}
	\delta(t) = \int_{0}^{t}G(t-s)B \left[\int_{0}^{1}\left(F'(\mu(s)) - F'(\mu_{0}(s)) \right)d\alpha \right] C\Delta(s)ds + \\ + \int_{0}^{t}G(t-s)BF'(C\varphi^{s}(v_{0}))C\delta(s)ds
	\end{split}
	\end{equation}
    For the first integral using Lemma \ref{LEM: DelayODEGeneratorLemma} and the H\"{o}elder inequality, for any $t \in [0,T]$ we have
    \begin{equation}
    	\label{EQ: DelayDiffEsimate1}
    	\begin{split}
    		\left| \int_{0}^{t}G(t-s)B \left[\int_{0}^{1}\left(F'(\mu(s)) - F'(\mu_{0}(s)) \right)d\alpha \right] C\Delta(s) \right| \leq \\ \leq M_{A}e^{\varkappa_{0} T} \cdot \|B\| \cdot \left(\int_{0}^{T} \int_{0}^{1}|F'(\mu(s)) - F'(\mu_{0}(s)|^{2}d\alpha ds\right)^{1/2} \times \\
    	    \times \left(\int_{0}^{T}|C\Delta(s)|^{2}ds\right)^{1/2}.
    	\end{split}
    \end{equation}
    From Lemma \ref{LEM: DelayMeasuEstimate} we have that $\mu(\cdot) \to \mu_{0}(s)$ in $L_{2}(0,T;\mathbb{R}^{r})$ as $h \to 0+$ uniformly in $\xi \in \mathcal{B}$. Now using the Tonelli theorem, we may apply Lemma \ref{LEM: DelayEquationsBoundedOperatorContinuity} to conclude that the limit
    \begin{equation}
    	R(h;\xi,v_{0}) :=  \left(\int_{0}^{T} \int_{0}^{1}|F'(\mu(s)) - F'(\mu_{0}(s)|^{2}d\alpha ds\right)^{1/2} \to 0 \text{ as } h \to 0
    \end{equation}
    is uniform in $\xi \in \mathcal{B}$.
    
    Since $\Delta(\cdot) \in \mathcal{H}^{T}_{2}$ (see Remark \ref{REM: VariationOfConstantsDelay}), from \textbf{(MES*)} and \textbf{(ULIP)} we get a constant $M'=M'(T)>0$ such that
    \begin{equation}
    	\label{EQ: DelayDiffEsimate2}
    	\left(\int_{0}^{T}|C \Delta(s) |^{2}ds\right) \leq M' |h\xi|_{\mathbb{H}}. 
    \end{equation}
    
    For the second integral in \eqref{EQ: DelayDiffDeltaEstimate} we have
    \begin{equation}
    	\label{EQ: DelayDiffEsimate3}
    	\left| \int_{0}^{t}G(t-s)BF'(C\varphi^{s}(v_{0}))C\delta(s)ds \right| \leq M_{A} e^{\varkappa_{0} T} \cdot \|B\| \cdot \int_{0}^{t}C\delta(s)ds.
    \end{equation}
    Since $\delta(0)=0$, from \textbf{(MES)} we get that $\int_{0}^{t}|C\delta(s)|ds \leq M_{\widetilde{C}}\int_{0}^{t}|\delta(s)|_{\mathbb{H}}ds$. Combining this and \eqref{EQ: DelayDiffEsimate1}, \eqref{EQ: DelayDiffEsimate2}, \eqref{EQ: DelayDiffEsimate3} with \eqref{EQ: DelayDiffDeltaEstimate}, we get for some constant $M''=M''(\mathcal{B},T,v_{0})>0$ that
    \begin{equation}
    	|\delta(t)|_{\mathbb{H}} \leq M'' \cdot R(h;\xi,v_{0}) \cdot |h \xi|_{\mathbb{H}} + M'' \cdot \int_{0}^{t}|\delta(s)|_{\mathbb{H}}ds.
    \end{equation}
    Now applying the Gronwall inequality, we finish the proof.
\end{proof}

Let $\mathcal{K}$ be a compact invariant set. As in \cite{Temam1997,ChepyzhovIlyin2004}, we say that the family $\varphi^{t}$ is \textit{quasi-differentiable} on $\mathcal{K}$ w.~r.~t. the family of quasi-differentials $L(t;v) \in \mathcal{L}(\mathbb{H})$, where $t \geq 0$ and $v \in \mathcal{K}$, if for all $t \geq 0$ 
\begin{equation}
\label{EQ: QuasiDifferentiabilityLimit}
\sup_{\substack{v_{1},v_{2} \in \mathcal{K} \\ |v_{1}-v_{2}|_{\mathbb{H}} \leq \varepsilon }}\frac{\left|\varphi^{t}(v_{2})-\varphi^{t}(v_{1})-L(t;v_{1})(v_{2}-v_{1}) \right|_{\mathbb{H}}}{\left|v_{1}-v_{2}\right|_{\mathbb{H}}} \to 0 \text{ as } \varepsilon \to 0+
\end{equation}
and the following properties are satisfied
\begin{description}
	\item[\textbf{(QD1)}] $\sup\limits_{\substack{t \in [0,1]\\v \in \mathcal{K}}}\|L(t;v)\| < \infty$;
	\item[\textbf{(QD2)}] $L(t+s;v) = L(t;\varphi^{s}(v))L(s;v)$ for all $t,s \geq 0$ and $v \in \mathcal{K}$.
\end{description}

\begin{theorem}
	\label{TH: DifferentiabilityTh}
	Suppose that $F$ is independent of $t$, $F \in C^{1}(\mathbb{R}^{r};\mathbb{R}^{m})$ and $F'$ is globally bounded. Consider a compact set $\mathcal{K} \subset \mathbb{H}$, which is invariant w.~r.~t. the semiflow $\varphi^{t}$, $t \geq 0$, given by equation \eqref{EQ: AbstractDelayHilberSpace}. Then the family $\varphi^{t}$ is quasi-differentiable w.~r.~t. the family of quasi-differentials $L(t;v)$ given by
	\begin{equation}
	\label{EQ: DelayQuasiDiffDef}
	L(t;v)\xi := V(t;v;\xi),
	\end{equation}
	where $V(t)=V(t;v_{0};\xi)$ is the solution to \eqref{EQ: LinearizedEquation} with $V(0)=\xi$. Moreover, the map $v_{0} \mapsto L(t,v_{0})$ is continuous as a map from $\mathcal{H}$ to $\mathcal{L}(\mathbb{H})$ for all $t \geq 0$.
\end{theorem}
\begin{proof}
	If the linear operator $L(t;v) \colon \mathbb{H} \to \mathbb{H}$ is defined as in \eqref{EQ: DelayQuasiDiffDef}, Theorem \ref{LEM: DelayDifferentiabilityLemma} guarantees that \eqref{EQ: QuasiDifferentiabilityLimit} is satisfied for all $t \geq 0$ since it shows that $L(t;v_{0})$ is the Fr\'{e}chet differential of $\varphi^{t}$ at $v_{0}$. From \textbf{(ULIP)} of Theorem \ref{TH: DelaySemigroupTh} it follows that $L(t;v)$ is a bounded linear operator and \textbf{(QD1)} is satisfied. The property in \textbf{(QD2)} is equivalent to 
	\begin{equation}
    V(t+s;v;\xi)=V(t;\varphi^{s}(v);V(s;v;\xi)) \text{ for all } t,s \geq 0, v \in \mathcal{K}, \xi \in \mathbb{H},
	\end{equation}
    but this is just the uniqueness of solutions to \eqref{EQ: LinearizedEquation} given by Lemma \ref{LEM: LinearizatinoWellPosed}. To show that the map $\mathbb{H} \ni v_{0} \mapsto L(t;v_{0}) \in \mathcal{L}(\mathbb{H})$ is continuous for every $t \geq 0$ we define $\delta(t) := L(t;v)\xi - L(t;v_{0})\xi = V(t;v;\xi)-V(t;v_{0};\xi)$ for some fixed $v_{0} \in \mathbb{H}$ and arbitrary $v \in \mathbb{H}$. Let also $\xi \in \mathbb{H}$ be such that $|\xi|_{\mathbb{H}} \leq 1$. Then $\delta(\cdot)$ satisfies the equation
    \begin{equation}
    \begin{split}
    \delta(t) = \int_{0}^{t}G(t-s)B(F'(C\varphi^{s}(v))CV(t;v;\xi) - F'(C\varphi^{s}(v_{0})) CV(t;v_{0};\xi))ds =\\= \int_{0}^{t}G(t-s)B(F'(C\varphi^{s}(v)) - F'(C\varphi^{s}(v_{0}))) CV(t;v_{0};\xi))ds +\\+ \int_{0}^{t}G(t-s)BF'(C\varphi^{s}(v))C\delta(s)ds.
    \end{split}
    \end{equation}
    Using the Gronwall lemma argument as in Theorem \ref{LEM: DelayDifferentiabilityLemma}, one can show that $v \to v_{0}$ implies that $\delta(t) \to 0$ uniformly in $\xi$ with $|\xi|_{\mathbb{H}} \leq 1$. But this gives the desired continuity. The proof is finished.
\end{proof}

The continuity of the map $v_{0} \to L(t;v_{0})$ plays an important role in dimension estimates. It is shown by V.V.~Chepyzhov and A.A.~Ilyin \cite{ChepyzhovIlyin2004} that this property is the only missing ingredient that makes the Hausdorff dimension estimate obtained by P.~Constantin, C.~Foias and R.~Temam \cite{Temam1997} hold for the fractal dimension also. This property is also essential for Theorem \ref{TH: MalletParetTh} below.

Note that any invariant compact $\mathcal{K}$, thanks to \textbf{(ULIP)}, must lie in $\mathbb{E}$  and the topologies of $\mathbb{E}$ and $\mathbb{H}$ on $\mathcal{K}$ coincide. Moreover, from \textbf{(ULIP)} we get that $\mathcal{K}$ is an image of itself under the Lipschitz map $\varphi^{2\tau}$. Therefore, the Hausdorff and fractal dimensions of $\mathcal{K}$ w.~r.~t. the metrics of $\mathbb{H}$ and $\mathbb{E}$ coincide. Applying a theorem of J.~Mallet-Paret \cite{MalletParet1976} (see also Section 4.6 in \cite{KuzReit2020}) to the compact Fr\'{e}chet differentiable map $\varphi^{2\tau}$ we get the following theorem.
\begin{theorem}
	\label{TH: MalletParetTh}
	Under the hypotheses of Theorem \ref{TH: DifferentiabilityTh} the Hausdorff dimension of $\mathcal{K}$ is finite.
\end{theorem}
Moreover, from \cite{ChepyzhovIlyin2004} we can show finiteness of the fractal dimension that is not less than the Hausdorff dimension (see Remark 3.6 in \cite{KuzReit2020}). For this we need some preparations.

For a compact linear operator $L \colon \mathbb{H} \to \mathbb{H}$ let $\alpha_{1} \geq \alpha_{2} \geq \alpha_{3} \geq \ldots$ be its singular values. For $d = k + s$, where $k$ is a non-negative integer and $s \in (0, 1]$, we consider the singular value function of $L$ defined as
\begin{equation}
\omega_{d}(L):=\alpha_{1} \cdot \ldots \cdot \alpha_{k} \cdot \alpha^{s}_{k+1} = \omega^{1-s}_{k}\omega^{s}_{k+1}.
\end{equation}

From Theorem 2.1 in \cite{ChepyzhovIlyin2004} applied to the map $\varphi^{t}$ for some fixed $t \geq 2\tau$ we get the following theorem.
\begin{theorem}
	\label{TH: ChepyzhovIlyinCor}
	Under the hypotheses of Theorem \ref{TH: DifferentiabilityTh} suppose that for some $t \geq 2 \tau$ and $d>0$ we have
	\begin{equation}
	\label{EQ: SingularValuesSqueezing}
	\sup\limits_{v \in \mathcal{K}}\omega_{d}(L(t;v)) < 1.
	\end{equation}
	Then for the fractal dimension of $\mathcal{K}$ in $\mathbb{H}$ we have $\operatorname{dim}_{F}\mathcal{K} < d$.
\end{theorem}
The main result of \cite{ChepyzhovIlyin2004} and the above definitions can be also formulated for the case of a non-compact operator $L$ \cite{Temam1997}. In our situation, the compactness of $L(t;v)$ for $t \geq 2 \tau$ allows us to stay in the compact context, which is simpler. Since $L(t;v)$ is a compact operator for $t \geq 2\tau$, its singular values tend to zero and therefore $\omega_{d}(L(t;v)) < 1$ is satisfied for all sufficiently large $d > 0$. Using the compactness of $\mathcal{K}$ and the continuity of $L(t,v)$ (and, consequently, $\omega_{d}(L(t;v))$) w.~r.~t. $v \in \mathcal{K}$, one can show that for every $t\geq 2\tau$ the inequality in \eqref{EQ: SingularValuesSqueezing} is satisfied for some sufficiently large $d > 0$. Thus, we immediately have the following theorem which strengthens Theorem \ref{TH: MalletParetTh}.
\begin{theorem}
	\label{TH: ChepyzhovIliynCor}
	Under the hypotheses of Theorem \ref{TH: DifferentiabilityTh} the fractal dimension of $\mathcal{K}$ is finite.
\end{theorem}

In particular, from Theorem \ref{TH: ChepyzhovIliynCor} and the Lipschitz property of the semiflow we immediately have finiteness of the topological entropy of $\varphi^{t}$ restricted to $\mathcal{K}$ (see, for example, \cite{KuzReit2020}).

It is interesting to obtain effective dimension estimates and this seems to be a nontrivial problem. Let us start a discussion by formulating the following analog of the Liouville trace formula for \eqref{EQ: LinearizedEquation}.
\begin{lemma}
	\label{LEM: TraceFormulaDelay}
	Under the hypotheses of Theorem \ref{TH: DifferentiabilityTh} let $k > 0$ be an integer and $v_{0} \in \mathbb{H}$ be fixed. Then for any $\xi_{j} \in \mathcal{D}(A)$ and $V_{j}(t)=V(t;v_{0},\xi_{j})$, where $j = 1,\ldots, k$ and $t \geq 0$, we have
	\begin{equation}
	\label{EQ: TraceFormula}
	\begin{split}
		\left| V_{1}(t) \wedge \ldots \wedge V_{k}(t) \right|_{\bigwedge^{k}\mathbb{H}} =\\= \left| \xi_{1} \wedge \ldots \wedge \xi_{k} \right|_{\bigwedge^{k}\mathbb{H}} \operatorname{exp}\left( \int_{0}^{t} \operatorname{Tr}( (A+BF'(C\varphi^{s}(v_{0}) )C) \circ \Pi(s)) ds \right),
	\end{split}
	\end{equation}
	where $\Pi(s)$ is the orthogonal projector onto $\operatorname{Span}( V_{1}(s),\ldots,V_{k}(s))$.
\end{lemma}
\noindent Here $\bigwedge^{k}\mathbb{H}$ denotes the $k$-th exterior power of $\mathbb{H}$ endowed with the inner product $|\cdot|_{\bigwedge^{k}\mathbb{H}}$, which is given for $\eta_{1} \wedge \ldots \wedge \eta_{k} \in \bigwedge^{k}\mathbb{H}$ by the $k$-dimensional volume of the parallelepiped with edges $\eta_{1},\ldots,\eta_{k}$ (for a more precise treatment see, for example, \cite{Temam1997}). Since $\xi_{j} \in \mathcal{D}(A)$, we have that $V_{j}(\cdot)$ is a classical solution of \eqref{EQ: LinearizedEquation} and, in particular, $V_{j}(t) \in \mathcal{D}(A)$ for all $t \geq 0$. Thus, the operator under $\operatorname{Tr}$ is bounded, has finite-dimensional range and continuously depend on $s$. In particular, its trace is well-defined. After these explanations Lemma \ref{LEM: TraceFormulaDelay} can be proved in a standard way (see, for example, Subsection 2.3 of Chapter V from \cite{Temam1997}).

As above, let $L \colon \mathbb{H} \to \mathbb{H}$ be a compact operator. It is well-known (see Proposition 1.4 from Subsection 1.3 of Chapter V in \cite{Temam1997}) that for any integer $k>0$ we have
\begin{equation}
\label{EQ: SingularValuesFunctionAsNorm}
\omega_{k}(L) = \sup\limits_{\substack{\xi_{1},\ldots,\xi_{k} \in \mathbb{H} \\ |\xi_{j}|_{\mathbb{H}} \leq 1 \forall j }} \left| L\xi_{1} \wedge \ldots \wedge L\xi_{k} \right|_{\bigwedge^{k}\mathbb{H}} = \sup\limits_{\substack{\xi_{1},\ldots,\xi_{k} \in \mathcal{D}(A) \\ |\xi_{j}|_{\mathbb{H}} \leq 1 \forall j }} \left| L\xi_{1} \wedge \ldots \wedge L\xi_{k} \right|_{\bigwedge^{k}\mathbb{H}},
\end{equation}
where the last inequality is due to the density of $\mathcal{D}(A)$ in $\mathbb{H}$.

Using \eqref{EQ: SingularValuesFunctionAsNorm} and Lemma \ref{LEM: TraceFormulaDelay}, one usually provides some estimates for the trace of the operator in \eqref{EQ: TraceFormula}, which makes it possible to apply Theorem \ref{TH: ChepyzhovIlyinCor}. In the case of parabolic problems there is a general estimate for the trace of a positive self-adjoint operator (see Chapter VI in \cite{Temam1997} or Lemma 4.21 in \cite{CarvalhoLangaRobinson2012}), which in some cases (for example, for reaction-diffusion equations) overrides the effect of the nonlineary (that is $F$ in our case) and makes it possible to obtain effective dimension estimates. Linear operators corresponding to delay equations are not self-adjoint and thus it is not obvious how one may derive effective dimension estimates with the use of Theorem \ref{TH: ChepyzhovIlyinCor} in the general case. As we will see below, a naive approach concerned with the study of the symmetrized operator $(A+A^{*}) / 2$ does not work here, although it has proved to be effective for ODEs (see \cite{KuzReit2020}). We are familiar only with the paper of J.W-H.~So and J.~Wu \cite{SoWu1991}, where dimension estimates for delay equations through the trace estimates are considered (in \cite{SoWu1991} a class of parabolic equations with delay is studied). However, \cite{SoWu1991} ends with an abstract application of the Constantin-Foias-Temam estimate (in the same way we did in Theorem \ref{TH: ChepyzhovIliynCor}) and, as it is said in the introduction, its concrete realizations are left for forthcoming papers, but none of such papers seems to be appeared. Another example is the monograph of A.N.~Carvalho, J.A.~Langa and J.C.~Robinson \cite{CarvalhoLangaRobinson2012}, which in particular treats dimension estimates, contains a chapter devoted to delay equations, but there are no discussions concerned with effective dimension estimates for such equations. Thus, we consider it important to draw attention to the following problem.
\begin{problem}
	\label{PROB: ProblemEffectiveDimEstDelay}
	How to obtain effective dimension estimates for delay equations?
\end{problem}
A partial solution to the above problem can be given by construction of inertial manifolds \cite{Anikushin2020Geom,Anikushin2020Red,Chicone2003,Temam1997}. However, the existence of an inertial manifold only gives an integer estimate given by its dimension and such an estimate is too rough to apply, for example, some criteria of non-existence of periodic orbits based on dimension estimates (see R.A.~Smith \cite{Smith1986HD}; M.Y.~Li and J.S.~Muldowney \cite{LiMuldowney1995}). Such criteria can be mixed with various developments of the Poincar\'{e}-Bendixson theory for infinite-dimensional dynamical systems (see \cite{Anikushin2020Geom,MalletParetSell1996}) to obtain convergence properties. On the other hand, the existence of smooth inertial manifolds at least theoretically leads to a certain ODE given by a smooth vector field (the so-called inertial form), which describes dynamics on the manifold. The linearized vector field can be symmetrized to obtain effective dimension estimates. To the best of our knowledge, the only case, to which this approach can be efficiently applied, is given by delay equations with small delays, for which C.~Chicone in \cite{Chicone2003} obtained an expansion for the vector field by the delay value considered as a small parameter.

Note that in the case of ODEs there is an effective approach called the Leonov method (see N.V.~Kuznetsov \cite{Kuznetsov2016LeonovMethod}; G.A.~Leonov et al. \cite{LeonovetalExactDim2016}; N.V.~Kuznetsov and V.~Reitmann \cite{KuzReit2020}) concerned with changes of the metric tensor via special Lyapunov-like functions. Sometimes it allows to obtain sharp (or even exact \cite{LeonovetalExactDim2016}) estimates for the Lyapunov dimension (and, consequently, for the Hausdorff and fractal dimensions) on the global attractor. To the best of our knowledge, there is still no efficient applications of the Leonov method for infinite-dimensional systems. Its development with the synthesis of the above mentioned criteria for the nonexistence of periodic orbits \cite{Smith1986HD,LiMuldowney1995} has particular interest.

Let $A \colon \mathcal{D}(A) \subset \mathbb{H} \to \mathbb{H}$ be a closed linear operator. Let $\mathbb{L} \subset \mathcal{D}(A)$ be a finite dimensional subspace and $\Pi_{\mathbb{L}}$ be the orthogonal (in $\mathbb{H}$) projector onto $\mathbb{L}$. The \textit{trace} of $A$ on $\mathbb{L}$ is the value $\operatorname{Tr}(A \circ \Pi_{\mathbb{L}})$. Let us define the \textit{trace numbers} $\beta_{1}(A),\beta_{2}(A),\ldots$ by induction from the relations
\begin{equation}
\label{EQ: OperatorTraceBetaNumbers}
\beta_{1}(A)+\ldots+\beta_{k}(A) = \sup_{\substack{\mathbb{L} \subset \mathcal{D}(A)\\ \operatorname{dim}\mathbb{L}=k}} \operatorname{Tr}(A \circ \Pi_{\mathbb{L}}),
\end{equation}
where $k=1,2,\ldots$. In Theorem \ref{TH: TraceNumbersTheorem} below the trace numbers are related to the eigenvalues of a self-adjoint extension $S$ of the symmetrized operator $S_{0} := (A+A^{*})/2$. In the case of delay equations such operators as $S$ are defined everywhere in $\mathbb{H}$ and have null spaces of finite codimension. By considering some simplest cases, one can see that the trace numbers do not depend on the delay value and their sum can never be negative that makes them inappropriate for studying delay equations.

Let us illustrate the problem by means of the following example. We consider for $\alpha \in (0,1)$ and $\tau>0$ the operator $A$ in $\mathbb{H} = \mathbb{R} \times L_{2}(-\tau,0;\mathbb{R})$ given by
\begin{equation}
\label{EQ: ExampleOperatorBetaNumbers}
(x,\phi) \mapsto \left(\phi(0) - \alpha \phi(-\tau), \frac{\partial}{\partial \theta}\phi \right).
\end{equation}
Its domain $\mathcal{D}(A)$ consists of all $(x,\phi) \in \mathbb{H}$ such that $\phi \in W^{1,2}(-\tau,0;\mathbb{R})$ and $\phi(0)=x$. Such an operator arises in the example from Section \ref{SEC: ExampleDelay}.
\begin{lemma}
	\label{LEM: OperatorTraceBetaNumbersExample}
	For the operator $A$ given by \eqref{EQ: ExampleOperatorBetaNumbers} we have $\beta_{1}(A) = \frac{3+\alpha^{2}}{2}$ and $\beta_{k}(A) = 0$ for all $k \geq 2$.
\end{lemma}
\begin{proof}
	Let $\mathbb{L} \subset \mathcal{D}(A)$ be any $k$-dimensional subspace. Let $e_{i} = (\phi_{i}(0),\phi_{i}(\cdot)) \in \mathcal{D}(A)$, where $i=1,2,\ldots,k$, be an orthonormal basis for $\mathbb{L}$. Then we have
	\begin{equation}
	\label{EQ: TraceExampleSuarezLemma}
	\operatorname{Tr}( A \circ \Pi_{\mathbb{L}} ) = \sum_{i=1}^{k} (Ae_{i},e_{i})_{\mathbb{H}} = \sum_{i=1}^{k} \left( \frac{3}{2} \phi_{i}(0)^{2} - \alpha \phi_{i}(-\tau) \phi_{i}(0) - \frac{1}{2} \phi_{i}(-\tau)^{2} \right).
	\end{equation}
	Let us consider the case $k=1$. From \eqref{EQ: TraceExampleSuarezLemma} we have
	\begin{equation}
	\begin{split}
	\beta_{1}(A) &= \sup \left( \frac{3}{2} \phi(0)^{2} - \alpha \phi(-\tau) \phi(0) - \frac{1}{2} \phi(-\tau)^{2} \right) = \\ &= \sup \phi(0)^{2} \cdot \left( \frac{3}{2} - \alpha \frac{\phi(-\tau)}{\phi(0)} - \frac{1}{2} \cdot \left(\frac{\phi(-\tau)}{\phi(0)}\right)^{2} \right),
	\end{split}
	\end{equation}
	where the first supremum is taken over all $\phi \in W^{1,2}(-\tau,0;\mathbb{R})$ such that $|\phi(0),\phi(\cdot)|_{\mathbb{H}} = 1$ and the second has the additional constraint $\phi(0) \not= 0$. Since $\phi(0)^{2} \leq 1$ and the maximum value of the quadratic polynomial $-1/2 x^2 - \alpha x + 3/2$ is $(3+\alpha^{2})/2$, we have $\beta_{1}(A) \leq (3+\alpha^{2})/2$. But it is easy to construct a sequence of functions $\phi(\cdot)$ showing that $(3+\alpha^{2})/2$ is indeed the value of $\beta_{1}(A)$.
	
	In the case $k=2$ let us observe that the orthonormal basis $e_{1},e_{2}$ for $\mathbb{L}$ can always be chosen such that $\phi_{1}(0)=0$. Then from \eqref{EQ: TraceExampleSuarezLemma} we have
	\begin{equation}
	\beta_{1}(A)+\beta_{2}(A) = \sup \left( -\frac{1}{2}\phi_{1}(-\tau)^{2} + \frac{3}{2} \phi_{2}(0)^{2} - \alpha \phi_{2}(-\tau) \phi_{2}(0) - \frac{1}{2} \phi_{2}(-\tau)^{2} \right),
	\end{equation}
	where the supremum is taken over all $\phi_{1},\phi_{2} \in W^{1,2}(-\tau,0;\mathbb{R})$ such that $\phi_{1}(0)=0$, \\ $| \phi_{i}(0),\phi_{i}(\cdot) |_{\mathbb{H}} = 1$ and $\int_{-\tau}^{0} \phi_{1}(\theta)\phi_{2}(\theta) d\theta = 0$. It is clear that the supremum do not exceed $(3+\alpha^{2})/2$. But it is not hard to construct a sequence of functions (assuming also that $\phi_{1}(-\tau) = 0$), which shows that it is, in fact, equal to $(3+\alpha^{2})/2$. Thus, $\beta_{2}(A) = 0$.
	
	If $k \geq 2$ then we can choose the orthonormal basis $e_{1},\ldots,e_{k}$ in such a way that $\phi_{i}(0) = 0$ for $i=1,\ldots,k-1$. From similar as above arguments we have that $\beta_{k}(A) = 0$ for $k \geq 2$.
\end{proof}
Although for $A$ being self-adjoint, the supremum in \eqref{EQ: OperatorTraceBetaNumbers} is achieved on some eigenspace, it can be seen from the proof of Lemma \ref{LEM: OperatorTraceBetaNumbersExample} that for the operators corresponding to delay equations the supremum in \eqref{EQ: OperatorTraceBetaNumbers} is not achieved on any nice subspace and the limiting subspace consists of discontinuous functions.
\begin{remark}
	\label{REM: EigenvaluesBetaExampleDelay}
	Below, we prove Theorem \ref{TH: TraceNumbersTheorem}, which shows that ``the limiting subspace'' is an eigenspace for a self-adjoint extension of the symmetrized operator $S_{0}:=(A+A^{*})/2$. Let us illustrate this for $A$ from Lemma \ref{LEM: OperatorTraceBetaNumbersExample}. In this case $S_{0}$ can be extended to a bounded self-adjoint operator $S$ in $\mathbb{H}$ given by 
	\begin{equation}
		(x,\phi) \mapsto \left(\frac{3 + \alpha^{2}}{2} x, 0 \right) \text{ for all } (x,\phi) \in \mathbb{H}.
	\end{equation}
	It is clear that the trace numbers $\beta_{k}(A)$ calculated in Lemma \ref{LEM: OperatorTraceBetaNumbersExample} coincide with the eigenvalues of $S$ ordered by non-increasing.
\end{remark}
Note also that there are two leading real eigenvalues $\lambda_{1}>0$ and $\lambda_{2} < 0$ of the operator $A$ given by \eqref{EQ: ExampleOperatorBetaNumbers} (see Section \ref{SEC: ExampleDelay}). From the dichotomy of linear problems for delay equations we have that the inequality $\lambda_{1} + \lambda_{2} < 0$ determines a region in the space of parameters $(\tau,\alpha)$ where two-dimensional volumes are squeezed. But it seems impossible to reveal this region from the trace formula \eqref{EQ: TraceFormula} without directly referring to the spectral decomposition.

Thus, it seems that the trace numbers $\beta_{k}(A)$, which are useful for studying parabolic problems and ODEs, are not appropriate for delay equations.

We finish this section by proving the following theorem.
\begin{theorem}
	\label{TH: TraceNumbersTheorem}
	Let a closed operator $A$ be such that $\mathcal{D}(A) \cap \mathcal{D}(A^{*})$ is dense in $\mathbb{H}$. Suppose the operator $S_{0}=(A+A^{*})/2$, which is defined at least on $\mathcal{D}(A) \cap \mathcal{D}(A^{*})$, can be extended to a self-adjoint linear operator $S$ having a compact resolvent, $\mathcal{D}(S) \supset \mathcal{D}(A)$ and such that its eigenvalues $\alpha_{k}(S)$ can be ordered by non-increasing. Then for any $k=1,2,\ldots$ we have
	\begin{equation}
		\label{EQ: TraceNumbersTheorem}
		\beta_{1}(A) + \ldots + \beta_{k}(A) = \alpha_{1}(S) + \ldots + \alpha_{k}(S).
	\end{equation}
\end{theorem}
\begin{proof}
	For breviety, let $(\cdot,\cdot)$ denote the inner product in $\mathbb{H}$. Let us firstly prove that we have the equality
	\begin{equation}
		\label{EQ: LemmaTraceNumbersIdentity}
		(Ae,e) = (Se,e) \text{ for all } e \in \mathcal{D}(A).
	\end{equation}
    Indeed, the equality in \eqref{EQ: LemmaTraceNumbersIdentity} holds for any $e \in \mathcal{D}(A) \cap \mathcal{D}(A^{*})$. But then we have an analogous identity for the corresponding bilinear forms as 
    \begin{equation}
    	\label{EQ: LemmaTraceNimbersIdentityBilinear}
    	(Av,w) + (v,Aw) = 2(Sv,w) \text{ for all } v,w \in \mathcal{D}(A) \cap \mathcal{D}(A^{*}).
    \end{equation}
    Now take instead of $v$ a sequence $v_{m} \in \mathcal{D}(A) \cap \mathcal{D}(A^{*})$ converging to some $v \in \mathcal{D}(A)$ in $\mathbb{H}$ as $m \to +\infty$. Using the equalities $(Av_{m},w) = (v_{m},A^{*}w)$, $(Sv_{m},w) = (v_{m},Sw)$ and $\mathcal{D}(S) \supset \mathcal{D}(A)$, one can show that \eqref{EQ: LemmaTraceNimbersIdentityBilinear} holds with $v \in \mathcal{D}(A)$ and, as a consequence of the symmetry, $w \in \mathcal{D}(A)$. Now taking $v=w=e \in \mathcal{D}(A)$ in \eqref{EQ: LemmaTraceNimbersIdentityBilinear} yields \eqref{EQ: LemmaTraceNumbersIdentity}.
    
    By definition, $\beta_{1}(A) + \ldots + \beta_{k}(A)$ is given by the supremum over all linear $k$-dimensional subspaces $\mathbb{L} \subset \mathcal{D}(A)$ of the value $\operatorname{Tr}( A \circ \Pi_{\mathbb{L}})$. Let $\mathbb{L}$ be fixed and let $e_{1}, \ldots, e_{k}$ be its orthonormal basis. In virtue of \eqref{EQ: LemmaTraceNumbersIdentity} we have
    \begin{equation}
    	\label{EQ: LemmaSymmetrizedSum}
    	\operatorname{Tr}( A \circ \Pi_{\mathbb{L}} ) = \sum_{i=1}^{k} (Ae_{i},e_{i}) = \sum_{i=1}^{k} \left( S e_{i},e_{i} \right) \leq \sum_{i=1}^{k}\alpha_{i}(S),
    \end{equation}
    where the last inequality is due to the variational principle for self-adjoint operators as in Lemma 2.1, Chapter VI from \cite{Temam1997}. This inequality becomes an equality if $\mathbb{L}$ is the eigenspace corresponding to $\alpha_{1}(S),\ldots,\alpha_{k}(S)$. The proof is finished.
\end{proof}
It is however interesting, whether the sum from \eqref{EQ: TraceNumbersTheorem} can be negative for some $k$ in particular cases arising from delay equations. We do not know of any such example.
\section{Example}
\label{SEC: ExampleDelay}

We consider the Suarez-Schopf model \cite{Suarez1988} for El Ni\~{n}o--Southern Oscillation (ENSO), which is given by the following scalar delay equation:
\begin{equation}
\label{EQ: ElNinoSSmodel}
\dot{x}(t) = x(t) - \alpha x(t-\tau) - x^{3}(t),
\end{equation}
where $\alpha \in (0,1)$ and $\tau > 0$ are parameters. Let us put $\gamma:=\sqrt{1+\alpha}$ and define the sets $\mathcal{C}_{R}:=\{ \phi \in C([-\tau,0];\mathbb{R}) \ | \ \|\phi\|_{\infty} \leq \gamma + R \}$ for $R > 0$. We also use $\mathring{\mathcal{C}}_{R}$ to denote the interior of $\mathcal{C}_{R}$. If a semiflow in $C([-\tau,0];\mathbb{R})$ is given, by $\omega(\phi_{0})$ we denote the $\omega$-limit set of $\phi_{0} \in C([-\tau,0];\mathbb{R})$ w.~r.~t. this semiflow.

Let us show that \eqref{EQ: ElNinoSSmodel} is dissipative and generates a semiflow in $C([-\tau,0];\mathbb{R})$. This is contained in the following lemma.
\begin{lemma}
	\label{LEM: RegionEstimateSS}
	The set $\mathcal{C}_{R}$ is positively invariant w.~r.~t. solutions of \eqref{EQ: ElNinoSSmodel}. In particular, \eqref{EQ: ElNinoSSmodel} generates a semiflow $\varphi^{t}$, $t \geq 0$, in $C([-\tau,0];\mathbb{R})$. Moreover, $\omega(\phi_{0}) \subset \mathcal{C}_{0}$ for all $\phi_{0} \in C([-\tau,0];\mathbb{R})$.
\end{lemma}
\begin{proof}
	Since the closure of a positively invariant set is positively invariant and the intersection of positively invariants sets is also positively invariant, it is sufficient to show that $\mathring{\mathcal{C}}_{R}$ is positively invariant for all $R>0$. Suppose the opposite, i.~e. there exist an initial condition $\phi_{0} \in \mathring{\mathcal{C}}_{R}$ and a time $t > 0$ such that for the classical solution $x(\cdot;0,\phi_{0})$ we have $\varphi^{s}(\phi_{0})=x_{s} \in \mathring{\mathcal{C}}_{R}$ for all $s \in [0,t)$ and $|x(t)|=\gamma + R$. From \eqref{EQ: ElNinoSSmodel} we have
	\begin{equation}
		\label{EQ: ElNinoDissLemma}
		\begin{cases} \dot{x}(t) &\leq  \gamma + R + \alpha (\gamma + R) - (\gamma + R)^{3} < 0, \text{ if } x(t)=R+\gamma,\\
			\dot{x}(t) &\geq  -(\gamma + R) - \alpha (\gamma + R) + (\gamma + R)^{3} > 0, \text{ if } x(t)=-(R+\gamma).
		\end{cases}
	\end{equation}
	that leads to a contradiction.
	
	Since any solution does not leave $\mathcal{C}_{R}$ for some $R \geq 0$, it remains bounded and therefore can be extended for times up to $+\infty$ (see Theorem 3.1, Section 2.3 in \cite{Hale1977}). Thus, \eqref{EQ: ElNinoSSmodel} generates a semiflow in $C([-\tau,0];\mathbb{R})$.
	
	Now suppose there is $\phi_{0} \in C([-\tau,0];\mathbb{R})$ such that $\omega(\phi_{0}) \not\subset \mathcal{C}_{0}$. Then there is $R>0$ such that $\omega(\phi_{0}) \subset \mathcal{C}_{R}$ and $\omega(\phi_{0})\not\subset \mathcal{C}_{R-\varepsilon}$ for all $0<\varepsilon<R$. Since $\omega(\phi_{0})$ cannot entirely lie on the boundary of $\mathcal{C}_{R}$ (due to similar as in \eqref{EQ: ElNinoDissLemma} arguments), we can find $0<\varepsilon_{1}<R$ and $\phi_{1},\phi_{2} \in \omega(\phi_{0})$ such that $\phi_{1} \in \mathcal{C}_{R-\varepsilon_{1}}$ and $\phi_{2} \in \mathcal{C}_{R}$ with $\|\phi_{2}\| = \gamma + R$. Then there is a time moment $t_{0} > 0$ for which $\varphi^{t_{0}}(\phi_{0})$ is close to $\phi_{1}$, namely, $\varphi^{t_{0}}(\phi_{0}) \in \mathcal{C}_{R-\varepsilon_{0}}$ for some $\varepsilon_{0} < \varepsilon_{1}$. But since $\mathcal{C}_{R-\varepsilon_{0}}$ is positively invariant we have that $\varphi^{t}(\phi_{0})$ is separated from $\phi_{2}$ for all $t \geq t_{0}$ that contradicts to $\phi_{2} \in \omega(\phi_{0})$. Thus, $\omega(\phi_{0}) \subset \mathcal{C}_{0}$ and the proof is finished.
\end{proof}

It should be noted that from a theorem of J.~Mallet-Paret and G.R.~Sell \cite{MalletParetSell1996} it follows that for \eqref{EQ: ElNinoSSmodel} the $\omega$-limit set of any point $\phi_{0}$ satisfies the Poincar\'{e}-Bendixson trichotomy, i.~e. it can be either a stationary point, either a periodic orbit or a union of a set of stationary points and homoclinic and heteroclinic orbits connecting them. However, numerical experiments in \cite{Suarez1988,Anikushin2021SS} show that there is a region, for which there are no periodic orbits (see below). 

An elementary analysis shows that there are three stationary states $\phi^{+} \equiv \sqrt{1-\alpha}$, $\phi^{0} \equiv 0$ and $\phi^{-} \equiv -\sqrt{1-\alpha}$. There is also a one-dimensional unstable manifold for $\phi^{0}$ for any parameters $\alpha \in (0,1)$ and $\tau > 0$. When $\alpha$ and $\tau$ are relatively small, the stationary states $\phi^{+}$ and $\phi^{-}$ are asymptotically stable. These parameters correspond to the region of linear stability in \cite{Suarez1988}, for which no periodic orbits were observed. However, in our paper \cite{Anikushin2021SS} we numerically found hidden and unstable periodic orbits in \eqref{EQ: ElNinoSSmodel}, arising after homoclinic bifurcations, for parameters from the region of linear stability which are close to its boundary (called the \textit{neutral curve} in \cite{Suarez1988}). Moreover, in \cite{Anikushin2021SS} we justified that these parameters (unlike the parameters from the region of linear instability) are related to the irregular nature of ENSO since the discovered multistability is sensible to exterior periodic forces or noise.

Since the semiflow generated by \eqref{EQ: ElNinoSSmodel} is dissipative, there is a global attractor $\mathcal{K} \subset \mathcal{C}_{0}$. To make the developed theory applicable, let $g \colon \mathbb{R} \to \mathbb{R}$ be a function which coincides with $x^{3}$ on $[-(\gamma + R), \gamma + R ]$ for some $R>0$ and it is smoothly extended outside of $[-(\gamma + R), \gamma + R ]$ in such a way that $g \in C^{1}(\mathbb{R};\mathbb{R})$ and $g'$ is globally bounded. Clearly, the set $\mathcal{K}$ will be also invariant for the semiflow $\varphi^{t}$ generated by \eqref{EQ: ElNinoSSmodel} with $x^{3}$ changed to $g(x)$. As before, we consider $\mathcal{K}$ as a subset of $\mathbb{H} = \mathbb{R} \times L_{2}(-\tau,0;\mathbb{R})$ and $\mathcal{K}$ is also invariant w.~r.~t. the corresponding semiflow in $\mathbb{H}$ given by Theorem \ref{TH: DelaySemigroupTh} which we also denote by $\varphi$.

Let $v_{0} \in \mathcal{K}$, where $v_{0}=(y_{0},\phi_{0})$. Let $y(t)=x(t;0,x_{0})$ be the classical solution of \eqref{EQ: ElNinoSSmodel} with $y(0)=y_{0}$. For $\xi \in \mathcal{D}(A)$, where $\xi = (z_{0},\phi_{0})$ such that $\phi_{0} \in W^{1,2}(-\tau,0;\mathbb{R})$ and $\phi_{0}(0)=z_{0}$, we consider the solution $V(t)=V(t,v_{0},\xi)$ of \eqref{EQ: LinearizedEquation}, which in our case can be described as a solution to
\begin{equation}
\label{EQ: ElNiNoLinearized}
\begin{split}
\dot{z}(t) &= z(t) - \alpha z(t-\tau) - 3 y^{2}(t) z(t),\\
\dot{\phi}(t) &= \frac{\partial}{\partial \theta} \phi(t).
\end{split}
\end{equation}
Here $V(t)=(z(t),\phi(t))$ and $V(0)=(z_{0},\phi_{0})$.

Let $L(t;v)$, where $t \geq 0$ and $v \in \mathcal{K}$, be the quasi-differentials given by Theorem \ref{TH: DifferentiabilityTh} for $\varphi^{t}$. It is clear that the zero stationary point $\phi^{0}$ belongs to $\mathcal{K}$. From a naive look at \eqref{EQ: ElNiNoLinearized} one may suggest that (due to the presence of the term $-3y^{2}(t)z(t)$) the squeezing of $d$-dimensional volumes at $\phi^{0}$ (where $y(t) \equiv 0$) is the worst possible comparing with other points on the attractor $\mathcal{K}$. However, the presence of delay makes it not so obvious and we can only conjecture this (see below).

The roots of the linear part of \eqref{EQ: ElNinoSSmodel} are given by
\begin{equation}
\label{EQ: RootsSuarezDelay}
1 - \alpha e^{-\tau p} - p = 0.
\end{equation}
It can be verified that if $\alpha \in (0,1)$ and $\tau>0$ then there are always one positive real root $\lambda_{1}$, one negative real root $\lambda_{2}$ and the others roots are located to the left from $\lambda_{2}$. From the dichotomy of autonomous linear systems (see Theorem 4.1, Chapter 7 in \cite{Hale1977}) it follows that $\lambda_{1}+\lambda_{2}<0$ indicates the squeezing of $2$-dimensional volumes at the zero stationary state $\phi^{0}$. From the above observations we pose the following problem.
\begin{problem}
	\label{PROB: SSmodel2DSqueeze}
	Let $\lambda_{1}=\lambda_{1}(\alpha,\tau)$ and $\lambda_{2}=\lambda_{2}(\alpha,\tau)$, where $\alpha \in (0,1)$ and $\tau>0$, be the positive and the negative real roots of \eqref{EQ: RootsSuarezDelay}. Is it true that the condition $\lambda_{1} + \lambda_{2} < 0$ excludes the presence of periodic orbits and homoclinics in \eqref{EQ: ElNinoSSmodel}?
\end{problem}
\begin{remark}
	It can be shown that $\lambda_{1}+\lambda_{2}<0$ is equivalent to
	\begin{equation}
		\label{EQ: SSLambda1PlusLambda2Ineq}
		\tau < \frac{\log(\frac{1+\sqrt{1-\alpha^{2}}}{\alpha})}{\sqrt{1-\alpha^{2}}}.
	\end{equation}
\end{remark}
The inequality $\lambda_{1} + \lambda_{2} < 0$ indicates a large region in the space of parameters $(\alpha,\tau)$, which is included (but strictly smaller) in the region of linear stability from \cite{Suarez1988} (see Fig. \ref{SSmodelConvAndProblemEst}). When $\lambda_{1} + \lambda_{2} > 0$, there may exist periodic orbits and homoclinics as it is indicated in \cite{Anikushin2021SS}. However, the region $\lambda_{1} + \lambda_{2} < 0$ is contained (but strictly smaller) in the region of true stability (with a gradient-like behavior) numerically obtained in \cite{Anikushin2021SS}. Thus, \cite{Anikushin2021SS} provides numerical evidence that the answer to Problem \ref{PROB: SSmodel2DSqueeze} should be positive.

At the end of Section \ref{SEC: DelayDifferentiability} we justified that straightforward applications of the Liouville trace formula \eqref{EQ: TraceFormula} are not appropriate to study delay equations. 

Let us present another viewpoint onto the above problem. For this we consider the 2nd compound cocycle, say $\Xi_{(2)} = \{ \Xi^{t}_{(2)}(q,\cdot) \}$, where $t \geq 0$ and $q \in \mathcal{K}$, associated with \eqref{EQ: ElNiNoLinearized}. This cocycle $\Xi_{(2)}$ acts in the 2-nd exterior power $\bigwedge^{2}\mathbb{E}$ of $\mathbb{E}=C([-\tau,0];\mathbb{R})$ and it is determined by the formula $\bigwedge^{2}\mathbb{E} \ni \xi_{1} \wedge \xi_{2} \mapsto \Xi_{(2)}^{t}(q,\xi_{1} \wedge \xi_{2}) := L(t;q)\xi_{1} \wedge L(t;q)\xi_{2} \in \bigwedge^{2}\mathbb{E}$ for $\xi_{1},\xi_{2} \in \mathbb{E}$. Note that $\Lambda^{2}\mathbb{E}$ is endowed with the so-called injective cross norm $\|\cdot\|$ (see \cite{MalletParretNussbaum2013} for details) for which we have $\| \xi_{1} \wedge \xi_{2} \| = \|\xi_{1}\|_{\mathbb{E}} \cdot \|\xi_{2}\|_{\mathbb{E}}$. There exists a leading spectral interval of the form $[a_{1},b_{1}]$ or $(a_{1},b_{1}]$ in the Sacker-Sell spectrum of $\Xi_{(2)}$ over $\mathcal{K}$. This means that for every $\varepsilon>0$ there exists a constant $M>0$ such that
\begin{equation}
	\label{EQ: DecayAreaEstimateSSmodelE}
	\sup_{q \in \mathcal{K}}\| \Xi^{t}_{(2)}(q,\cdot) \| \leq M e^{(b_{1}+\varepsilon) t} \text{ for } t \geq 0.
\end{equation}
Clearly, $b_{1} < 0$ indicates the uniform squeezing of two-dimensional volumes over $\mathcal{K}$. Thus, one may ask how to estimate $b_{1}$.

Using the result of D.~Dragi\v{c}evi\'{c} \cite{Dragicevic2018} we can prove the following theorem which describes $b_{1}$ as a Lyapunov exponent over a certain ergodic measure. In our case, all ergodic measures are concentrated on equilibria and periodic orbits and, consequently, $b_{1}$ can be described in terms of the Floquet exponents or eigenvalues. 
\begin{remark}
	\label{REM: DelayExampleQuasiCompactCocycle}
	Before giving the proof, we note that in \cite{Dragicevic2018} it is important that the cocycle is quasi-compact and compact cocycles not necessarily satisfy this property (see below). Although, it seems that the main result of \cite{Dragicevic2018} holds for compact cocycles also since the quasi-compactness is used only to apply the Multiplicative Ergodic Theorem, which is also satisfied for compact cocycles. Nevertheless, it is worth discussing this property since the general approach is potentially applicable for neutral delay equations. In terms of \cite{Dragicevic2018}, quasi-compactness requires two quantities $\kappa(q)$ and $\lambda(q)$, describing certain exponents of the growth rates of the cocycle over ergodic components, to be strictly separated as $\kappa(q) < \lambda(q)$. In our case $\kappa(q) \equiv -\infty$ due to compactness and the question is whether $\lambda(q) > -\infty$. In our concrete situation it can be stated as ``Can two-dimensional volumes decay faster than any exponent?''. Unfortunately, this (i.~e. $\lambda(q) = -\infty$) may happen for general infinite-dimensional systems (including delay equations) and their compound cocycles since, for example, there are counterexamples to the Floquet theory in infinite dimensions (see, for example, A.~Eden, S.~Zelik, V.K.~Kalantarov \cite{EdenZelikKalantarov2013}). Thus, results of \cite{MalletParretNussbaum2013} on the Floquet theory, showing that there is a countable number of non-zero Floquet multipliers, along with the Poincar\'{e}-Bendixson trichotomy from \cite{MalletParetSell1996} are crucial here. Moreover, in the general case, the existence of a gap of rank $j$ in the Sacker-Sell spectrum for the 1st compound (linearization) cocycle) guarantees that the compound cocycles upto $j$th are quasi-compact. Thus, for $\Xi_{(2)}$ the problem of quasi-compactness can also be resolved via the existence of two-dimensional inertial manifolds discussed in Section \ref{SEC: ExampleDelayIM}.
\end{remark}
\begin{theorem}
	\label{TH: FloquetMultipliersAsExtremalPointSSmodel}
	In the above introduced notation, the right endpoint $b_{1}$ of the leading spectral interval for the cocycle $\Xi_{(2)}$ in $\bigwedge^{2}C([-\tau,0];\mathbb{R})$ over the global attractor $\mathcal{K}$ is realized as
	\begin{equation}
		b_{1} = \frac{\operatorname{ln}|\mu_{1}\mu_{2}|}{\sigma},
	\end{equation}
     where $\mu_{1}$ and $\mu_{2}$ are the first two Floquet multipliers over a certain $\sigma$-periodic orbit or a stationary point (treated as a period $1$ orbit).
\end{theorem}
\begin{proof}
	Below, when referring to certain results for discrete-time systems, we consider the discrete cocycle $A(q,n) := \Xi^{\tau n}_{(2)}(q,\cdot)$ in $\bigwedge^{2}C([-\tau,0];\mathbb{R})$ over the homeomorphism\footnote{For the uniqueness of backward extensions of solutions (which gives the homeomorphism property of $\varphi^{t}$ restricted to the global attractor) see Theorem 4.1, Section 3.4 in \cite{Hale1977}.} $f(\cdot):=\varphi^{\tau}(\cdot)$ on $\mathcal{K}$. Note that the operators $A(q,n)$ are compact.
	
	Let $\mu$ be an ergodic Borel measure for $f$ on $\mathcal{K}$. Since the support of any Borel invariant measure must lie in the set of recurrent points (see item 1, Proposition 4.1.18 in \cite{KatokHasselblatt1996}), the support of $\mu$ contain only stationary or periodic points of the semiflow $\varphi$ due to the Poincar\'{e}-Bendixson trichotomy (Theorem 2.1 in \cite{MalletParetSell1996}). From the Floquet theory developed in \cite{MalletParretNussbaum2013} (see Theorem 5.1 therein), it follows that for any periodic point $q \in \mathcal{Q}$ we have
	\begin{equation}
		\lambda(q):=\lim\limits_{n \to +\infty}\frac{\| A(q,n) \|}{n} > -\infty
	\end{equation}
	and, in particular, the cocycle $A(q,n)$ is quasi-compact (see Remark \ref{REM: DelayExampleQuasiCompactCocycle} and Section 2.2 in \cite{Dragicevic2018}).
	
	By Theorem 7 in \cite{Dragicevic2018}, one can take $\mu$ such that $b_{1}$ is realized as a Lyapunov exponent for $\mu$. Let $q_{0} = \phi_{0}$ be any point in the support, at which the Lyapunov exponent exists. Without loss of generality, we may assume that $q_{0}$ is $\sigma$-periodic. If $\mu_{1}$ and $\mu_{2}$ are the first two Floquet multipliers (their existence is guaranteed by Theorem 5.1 in \cite{MalletParretNussbaum2013}), then $\sigma b_{1} \leq \operatorname{ln}|\mu_{1} \mu_{2}|$. But it is clear that $\sigma \operatorname{ln}|\mu_{1} \mu_{2}| \leq b_{1}$. This finishes the proof.
\end{proof}
Thus, Theorem \ref{TH: FloquetMultipliersAsExtremalPointSSmodel} reduces Problem \ref{PROB: SSmodel2DSqueeze} to that any periodic orbit must be asymptotically stable (in the sense $|\mu_{1}\mu_{2}| < 1$) in the considered region.

For a real number $r$ we consider the roots $\lambda^{(r)}_{1}=\lambda^{(r)}_{1}(\alpha,\tau)$, $\lambda^{(r)}_{2}=\lambda^{(r)}_{2}(\alpha,\tau)$, \ldots enumerated in decreasing (by the real part) order according to their multiplicity of the equation
\begin{equation}
	\label{EQ: RadiiEq}
	1-3r^{2}_{0} - \alpha e^{-\tau p } - p = 0.
\end{equation}

The following theorem provides an estimate for $|\mu_{1}\mu_{2}|$. It highly relates on the monotonicity results for $\Xi_{(2)}$ from \cite{MalletParretNussbaum2013}.
\begin{theorem}
	Let $y_{0}(\cdot)$ be a $\sigma$-periodic solution of \eqref{EQ: ElNinoSSmodel} and let $\mu_{1}$ and $\mu_{2}$ be the first two Floquet multipliers. Let $r_{0}:=\operatorname{max}_{t \in [0,\sigma]}|y_{0}(t)|$. Then for any $r \geq r_{0}$ we have
	\begin{equation}
		\label{EQ: SSMultiplicators2Estimate}
		\operatorname{ln}|\mu_{1} \mu_{2}| \leq \sigma \left[ 6 \cdot \frac{1}{\sigma}\int_{0}^{\sigma} \left( r^{2}-y^{2}_{0}(s) \right) ds + \lambda_{1}^{(r)}+\lambda^{(r)}_{2} \right],
	\end{equation}
    where $\lambda_{1}^{(r)}$ and $\lambda^{(r)}_{2}$ are the first two of roots of \eqref{EQ: RadiiEq}.
\end{theorem}
\begin{proof}
		Let $q_{0} \in \mathcal{K}$ be the $\sigma$-periodic point corresponding to $y_{0}(\cdot)$. Let $M_{q_{0}} := \Xi^{\sigma}_{(2)}(q_{0},\cdot)$ be the monodromy operator over $q_{0}$ for the 2nd compound cocycle $\Xi_{(2)}$. Note that the linearization along the orbit is given by
		\begin{equation}
			\label{EQ: SSmodelMultipEstimate1}
			\dot{z}(t) = (1-3y^{2}_{0}(t))z(t) - \alpha z(t-\tau).
		\end{equation}
		 We will compare $M_{q_{0}}$ with the time-$\sigma$ map for the 2nd compound process, say $M_{r}$, generated by
		 \begin{equation}
		 	\label{EQ: SSmodelMultipEstimate2}
		 	\dot{z}(t) = (1-3r^{2})z(t) - \alpha z(t-\tau).
		 \end{equation}
	     After the changes of variables given by $w(t)=\operatorname{exp}\left(\int_{0}^{t}(3y^{2}_{0}(s)-1)ds\right) z(t)$ and $w(t)=\operatorname{exp}\left(\int_{0}^{t}(3r-1)ds\right) z(t)$ equations \eqref{EQ: SSmodelMultipEstimate1} and \eqref{EQ: SSmodelMultipEstimate2} transfer respectively into
	     \begin{equation}
	     	\begin{split}
	     		\dot{w}(t) = -b_{1}(t) w(t-\tau),\\
	     		\dot{w}(t) = -b_{2}(t) w(t-\tau).
	     	\end{split}
	     \end{equation}
         Note that $b_{1}(t) = \alpha_{0} \int_{t-\tau}^{t}(3y^{2}_{0}(s)-1)ds$ and $b_{2}(t) = \alpha_{0} \int_{t-\tau}^{t}(3r^{2}-1)ds$. Let $\widetilde{M}_{q_{0}}$ and $\widetilde{M}_{r}$ denote the time-$\sigma$ maps of the corresponding 2nd compound processes in $\bigwedge^{2}C([-\tau,0];\mathbb{R})$. Since $b_{1} \leq b_{2}$, by repeating arguments from the proof of Proposition 5.3 in \cite{MalletParretNussbaum2013}, we have that $0 \leq \widetilde{M}_{q_{0}} \leq \widetilde{M}_{r}$ in the sense of the partial ordering given by a closed convex reproducing normal cone (the cone $K_{2}$ from formula (4.7) in \cite{MalletParretNussbaum2013}). By Proposition 5.7 in \cite{MalletParretNussbaum2013} we have that $\rho(\widetilde{M}_{q_{0}}) \leq \rho(\widetilde{M}_{r})$, where $\rho(\cdot)$ denotes the spectral radius.
         
         Now note that $\rho(\widetilde{M}_{q_{0}})= \operatorname{exp}\left(2\int_{0}^{\sigma} (3y^{2}(s)-1)ds \right) \rho(M_{q_{0}})$ and $\rho(\widetilde{M}_{r}) = \operatorname{exp}\left( 2\int_{0}^{\sigma} (3r^{2}-1)ds \right) \rho(M_{r})$. Since $\rho(M_{q_{0}})=|\mu_{1}\mu_{2}|$ and $\rho(M_{r}) = e^{\sigma( \lambda^{(r)}_{1} + \lambda^{(r)}_{2}) }$, we obtain the desired conclusion.
\end{proof}

\begin{remark}
	One can also obtain an estimate similar to \eqref{EQ: SSMultiplicators2Estimate} for any even number $m$ of the Floquet multipliers $\mu_{1},\mu_{2},\ldots,\mu_{m}$ as
	\begin{equation}
		\operatorname{ln}| \mu_{1}\mu_{2}\cdot\ldots\cdot\mu_{m} | \leq \sigma\left[ 3m \frac{1}{\sigma}\int_{0}^{\sigma} (r^{2}-y^{2}_{0}(s)) ds + \sum_{k=1}^{m} \lambda^{(r)}_{k} \right].
	\end{equation}
\end{remark}

From the point of view given by Lemma \ref{EQ: SSMultiplicators2Estimate}, it is required to provide a bound for periodic orbits. It turns out that the bound $\sqrt{1+\alpha}$ given by Lemma \ref{LEM: RegionEstimateSS} is not appropriate due to its roughness. The following lemma provides sharper estimates for the region containing the global attractor $\mathcal{K}$ for certain parameters.
\begin{lemma}
	\label{LEM: SharperRegionSS}
	Suppose that $\alpha \tau < 1/2$ and put 
	\begin{equation}
		C(\alpha,\tau):= \frac{\alpha \tau}{1-\alpha \tau} \cdot \frac{4}{3} \cdot (1-\alpha) \cdot \sqrt{ \frac{1-\alpha}{3}}.
	\end{equation}
    Let $R_{1}$ be the unique positive root of $-p^{3} + (1-\alpha )p +C(\alpha;\tau) = 0$ and put
    \begin{equation}
    	R_{2} := C(\alpha,\tau) \cdot \operatorname{max}\{ (1 - \alpha \tau)^{-1}, (\alpha\tau)^{-1}  \}.
    \end{equation}
    Then the global attractor $\mathcal{K}$ lies in the set $\mathcal{C}(R_{1};R_{2})$ defined by
    \begin{equation}
    	\label{EQ: SSmodelSharperRegionDef}
    	\mathcal{C}(R_{1};R_{2}) := \{ \phi \in C^{1}([-\tau,0];\mathbb{R}) \ | \ \|\phi\|_{\infty} \leq R_{1}, \| \phi' \|_{\infty} \leq R_{2}  \}.
    \end{equation}
\end{lemma}
\begin{remark}
	Note that for the radius $R_{1}$ from Lemma \ref{LEM: SharperRegionSS} we always have $R_{1} > \sqrt{1-\alpha}$. Moreover, in the region $\alpha \tau < 1/2$ the estimate for the radius of dissipativity given by the lemma significantly improves the estimate $\sqrt{1+\alpha}$ from Lemma \ref{LEM: RegionEstimateSS}, especially for larger $\alpha$ since $R_{1} \to \sqrt{1-\alpha}$ as $\alpha \to 1-$.
\end{remark}
\begin{proof}
	The proof is similar to Lemma \ref{LEM: RegionEstimateSS}, although contains more calculations. Let us consider the set $\mathcal{C}(R_{1};R_{2})$ as in \eqref{EQ: SSmodelSharperRegionDef} for certain $R_{1}$ and $R_{2}$ (not necessarily the ones specified in the statement of the lemma). Let $\mathring{\mathcal{C}}(R_{1};R_{2})$ be the set determined by strict inequalities and let us check its ``positive invariance'' for certain initial conditions and values $R_{1}$ and $R_{2}$. Suppose that $\phi_{0} \in \mathring{\mathcal{C}}(R_{1};R_{2})$. We suppose that the solution $x(t)=x(t;\phi_{0})$ satisfies $\dot{x}(0)=\phi'_{0}(0)$ and there exists $t_{0}>0$ such that $x_{t} \in \mathring{\mathcal{C}}(R_{1};R_{2})$ for all $t \in [0,t_{0})$ and $x_{t_{0}} \in \mathcal{C}(R_{1};R_{2})$. There are only two possible cases (when one of the strict inequalities is violated).

    Case 1: $|x(t_{0})| = R_{1}$. It is sufficient to assume that $x(t_{0})=R_{1}$ and $|x'(t)| \leq R_{2}$ for $t \in [0,t_{0}]$. From \eqref{EQ: ElNinoSSmodel} we have
    \begin{equation}
    	\dot{x}(t_{0}) = R_{1} - R^{3}_{1} - \alpha x(t_{0}-\tau).
    \end{equation}
    Note that $x(t_{0}-\tau)  \geq R_{1} - \tau R_{2}$. Thus, to get the impossibility of our situation it is sufficient to require that
    \begin{equation}
    	\label{EQ: SharperRegionSSmodelCase1}
    	(1-\alpha)R_{1}-R^{3}_{1} + \alpha \tau R_{2} < 0.
    \end{equation}
    
    Case 2: $|x'(t_{0})| = R_{2}$. It is sufficient to assume that $x'(t_{0}) = R_{2}$ and $|x(t)| \leq R_{1}$ for $t \in [0,t_{0}]$. From \eqref{EQ: ElNinoSSmodel} we have
    \begin{equation}
    	R_{2} = x(t_{0}) - \alpha x(t_{0}-\tau) - x^{3}(t_{0}) = (1-\alpha) x(t_{0}) + \alpha \int_{t_{0}-\tau}^{t_{0}}\dot{x}(s)ds - x^{3}(t_{0}).
    \end{equation}
    In particular, we have $R_{2} \leq (1-\alpha)x(t_{0})-x^{3}(t_{0}) + \alpha \tau R_{2}$. For $x(t_{0}) \geq 0$ we have that 
    \begin{equation}
    	(1-\alpha)x(t_{0}) - x^{3}(t_{0}) \leq (1-\alpha) \cdot \sqrt{\frac{1-\alpha}{3}} \cdot \frac{4}{3}
    \end{equation}
    and, consequently, it is sufficient to require
    \begin{equation}
    	R_{2} > (\alpha \tau)^{-1} \cdot C(\alpha;\tau)
    \end{equation}
    to get the desired impossibility. On the other hand, if $x(t_{0}) < 0$, it is sufficient to require that
    \begin{equation}
    	R_{2} > (1-\alpha \tau)^{-1} \cdot ( -(1-\alpha)R_{1} + R^{3}_{1}  ).
    \end{equation}
    Let us take $R_{0}>0$ and $R_{2} := (\alpha \tau)^{-1} \cdot C(\alpha,\tau) + R_{0}$. Then \eqref{EQ: SharperRegionSSmodelCase1} takes the form
    \begin{equation}
    	\label{EQ: SharperEstimateSSI}
    	(1-\alpha) R_{1} - R^{3}_{1} +C(\alpha,\tau) + \alpha \tau R_{0} < 0.
    \end{equation}
    Moreover, for $R_{2}:=(1-\alpha \tau)^{-1} (-(1-\alpha)R_{1}+R^{3}_{1}) + R_{0}$ \eqref{EQ: SharperRegionSSmodelCase1} takes the form
    \begin{equation}
    	\label{EQ: SharperEstimateSSII}
    	(1-\alpha) R_{1} - R^{3}_{1} + \frac{\alpha \tau}{1-\alpha \tau} \cdot \left( -(1-\alpha)R_{1} + R^{3}_{1} \right) + \alpha \tau R_{0} < 0
    \end{equation}
    Note that $\alpha \tau < (1-\alpha \tau)$ is equivalent to $\alpha \tau < 1/2$. Thus, for any $R_{0} > 0$ there exists $R_{1}=R_{1}(R_{0})$ such that both \eqref{EQ: SharperEstimateSSI} and \eqref{EQ: SharperEstimateSSII} are satisfied and the corresponding region $\mathring{\mathcal{C}}(R_{1},R_{2})$ will be positively invariant in the above given sense if we take $R_{2} = \operatorname{max}\{ (\alpha \tau)^{-1} \cdot C(\alpha,\tau), (1-\alpha\tau)^{-1}(-(1-\alpha)R_{1} + R^{3}_{1}) \} + R_{0}$. Note that as $R_{0} \to 0+$ the value $R_{1}(R_{0})$ can be taken such that it tends to the value defined in the statement of the lemma.
    
    Now the final statement about the global attractor can be shown analogously to the corresponding part of Lemma \ref{LEM: RegionEstimateSS}.
\end{proof}

Now we use the obtained in Theorem \ref{LEM: SharperRegionSS} sharper bounds for the global attractor $\mathcal{K}$ to get the nonexistence of periodic orbits and homoclincs for certain parameters.
\begin{theorem}
	Suppose that $\alpha \tau < 1/2$ and for the radius $R_{1}$ from Lemma \ref{LEM: SharperRegionSS} we have
	\begin{equation}
		\label{EQ: SSmodelRoughMuEstimate}
		6 \cdot R^{2}_{1} + \lambda^{(R_{1})}_{1} + \lambda^{(R_{1})}_{2} < 0.
	\end{equation}
    Then $b_{1} < 0$ and there are no periodic orbits and homoclinics in \eqref{EQ: ElNinoSSmodel}. Moreover, any point tends to one of the stationary states $\phi^{0}$,$\phi^{+}$,$\phi^{-}$.
\end{theorem}
\begin{proof}
	Indeed, \eqref{EQ: SSmodelRoughMuEstimate} and \eqref{EQ: SSMultiplicators2Estimate} imply that for any periodic orbit the first two Floquet multipliers satisfy $|\mu_{1}\mu_{2}| < 1$. Note that $\alpha \tau < 1/2$ implies that $\lambda_{1} + \lambda_{2} < 0$ for the roots of \eqref{EQ: RootsSuarezDelay} and that $\phi^{+}$ and $\phi^{-}$ are asymptotically stable. Thus, Theorem \ref{TH: FloquetMultipliersAsExtremalPointSSmodel} implies that $b_{1} < 0$ and we have the uniform decay for two-dimensional volumes over the global attractor $\mathcal{K}$ in the sense of \eqref{EQ: DecayAreaEstimateSSmodelE}.
	
    Due to the smoothing estimate in \textbf{(ULIP)} from Theorem \ref{TH: DelaySemigroupTh}, one can show that an analog of \eqref{EQ: DecayAreaEstimateSSmodelE} holds also in the norm of $\mathbb{H}$ and from \eqref{EQ: SingularValuesFunctionAsNorm} we get that \eqref{EQ: SingularValuesSqueezing} is satisfied for $d=2$ and sufficiently large $t>0$. From Theorem \ref{TH: ChepyzhovIlyinCor} it follows that $\operatorname{dim}_{F}\mathcal{K} < 2$. Then Corollary 2 from \cite{LiMuldowney1995} implies that there are no periodic orbits or homoclinics in $\mathcal{K}$. Now the Poincar\'{e}-Bendixson trichotomy (Theorem 2.1 in \cite{MalletParetSell1996}) guarantees that the $\omega$-limit set of any point must be a single stationary point. The proof is finished.
\end{proof}
Fig. \ref{SSmodelConvAndProblemEst} shows the region determined by \eqref{EQ: SSmodelRoughMuEstimate}, however it is much smaller than the expected region considered in Problem \ref{PROB: SSmodel2DSqueeze}.

\begin{figure}
	\centering
	\includegraphics[width=1.\linewidth]{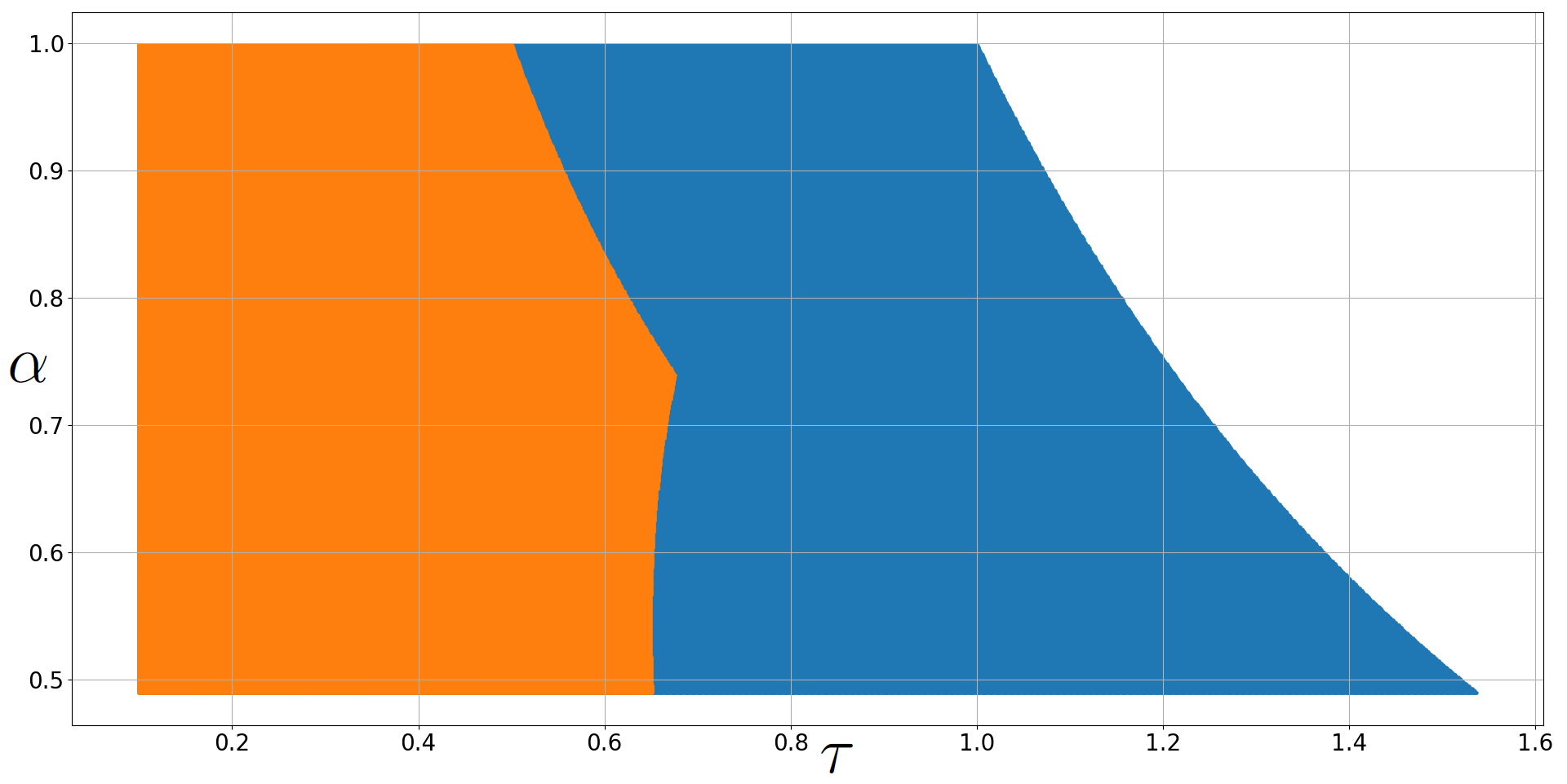}	
	\caption{A numerically obtained region in the space of parameters $(\tau,\alpha)$ of system \eqref{EQ: ElNinoSSmodel}, where \eqref{EQ: SSmodelRoughMuEstimate} is satisfied (orange) and \eqref{EQ: SSLambda1PlusLambda2Ineq} is satisfied (orange and blue).}
	\label{SSmodelConvAndProblemEst}
\end{figure}
\section{Inertial manifolds}
\label{SEC: InvariantManifoldsDelay}
For $v \in \mathbb{E}$ and $t \geq 0$ let us define $\psi^{t}(v,\xi):=L(t,v)\xi$, where $L(t,v)\xi$, as in Section \ref{SEC: DelayDifferentiability}, is a solution to the linearized equation \eqref{EQ: LinearizedEquation}. From Theorem 2.2, Chapter 2 in \cite{Hale1977}, we immediately have the following lemma.
\begin{lemma}
	\label{LEM: DelayLinearCocycle}
	Let $F \in C^{1}(\mathbb{R}^{r};\mathbb{R}^{m})$ and $F'$ be globally bounded. Then the map $(t,v,\xi) \mapsto \psi^{t}(v,\xi)$ is continuous as a map from $\mathbb{R}_{+} \times \mathbb{E} \times \mathbb{E}$ to $\mathbb{E}$.
\end{lemma}
Now let $\mathfrak{A} \subset \mathbb{E}$ be an invariant w.~r.~t. the semiflow $\varphi$ from Section \ref{SEC: DelayDifferentiability} finite-dimensional topological manifold such that the semiflow is invertible on $\mathfrak{A}$ and, consequently, in virtue of the Brouwer theorem on invariance of domain, it defines a flow on $\mathfrak{A}$. We also put $\vartheta^{t}(v) := \restr{\varphi^{t}}{\mathfrak{A}}(v)$ for $t \in \mathbb{R}$, $v \in \mathfrak{A}$ and $\mathcal{Q}:=\mathfrak{A}$. Then $(\mathcal{Q},\vartheta)$ is a flow. It is obvious that we have the cocycle property satisfied as
\begin{equation}
\psi^{t+s}(v,\xi) = \psi^{t}(\vartheta^{s}(v),\psi^{s}(v,\xi)), \text{ for all } t,s \geq 0, v \in \mathfrak{A}, \xi \in \mathbb{E}.
\end{equation}
From this and Lemma \ref{LEM: DelayLinearCocycle} it follows that $\psi$ is a cocycle in $\mathbb{E}$ over the base flow $\vartheta=\varphi$ in $\mathcal{Q} = \mathfrak{A}$. This allows us to apply results from Appendix \ref{APP: AdditionToReduction}.

Now we suppose that $F \in C^{1}$ with $F'$ globally bounded and the semiflow $\varphi$ (the same as in Section \ref{SEC: DelayDifferentiability}) satisfies conditions \textbf{(H1)},\textbf{(H2)} and \textbf{(H3)} from Appendix \ref{APP: AdditionToReduction}, i.~e. there exists an operator $P \in \mathcal{L}(\mathbb{E};\mathbb{E}^{*})$ such that for $V(v):=\langle v, Pv \rangle$ and some $\delta>0,\nu>0$ we have for all $v_{1},v_{2} \in \mathbb{E}$ and $r \geq 0$ (in terms of \textbf{(H3)} here we also put $\tau_{V} = 0$ for convenience) that
\begin{equation}
\label{EQ: SqueezingDelayFlow}
e^{2\nu r}V(\varphi^{r}(v_{1})-\varphi^{r}(v_{2})) - V(v_{1}-v_{2}) \leq -\delta \int_{0}^{r} e^{2\nu s}|\varphi^{s}(v_{1})-\varphi^{r}(v_{2})|^{2}_{\mathbb{H}}ds
\end{equation}
and there exists a splitting $\mathbb{E}=\mathbb{E}^{+} \oplus \mathbb{E}^{-}$ with $\dim \mathbb{E}^{-}=j$ and such that $P$ is positive on $\mathbb{E}^{+}$ and negative on $\mathbb{E}^{-}$. Moreover, $\mathbb{E}^{+}$ and $\mathbb{E}^{-}$ can be assumed to be $V$-orthogonal in the sense that $V(v) = V(v^{+}) + V(v^{-})$ for all $v \in \mathbb{E}$, where $v=v^{+}+v^{-}$ is the unique decomposition with $v^{+} \in \mathbb{E}^{+}$ and $v^{-} \in \mathbb{E}^{-}$ (see Appendix \ref{APP: AdditionToReduction}). Such an operator $P$ can be obtained with the aid of the Frequency Theorem (see \cite{Anikushin2020Freq} for a proof and detailed discussions which especially concern delay equations).
\begin{remark}
	In this section we discuss inertial manifolds in the classical phase space $\mathbb{E}$ (of course, we can work directly in $\mathbb{H}$). However, it seems crucial (at least from the point of view given by the Frequency Theorem \cite{Anikushin2020Freq}) to consider the weaker norm $|\cdot|_{\mathbb{H}}$ in the right-hand side of \eqref{EQ: SqueezingDelayFlow}.
\end{remark}

We will use the $V$-orthogonal projector $\Pi \colon \mathbb{E} \to \mathbb{E}^{-}$ defined by $Pv:=v^{-}$. From Theorem \ref{TH: DelaySemigroupTh} we have properties \textbf{(ULIP)} and \textbf{(COM)} satisfied for $\varphi$ and thus we may apply Theorem \ref{TH: ReductionTheoremGener} to get an invariant $j$-dimensional topological manifold $\mathfrak{A} \subset \mathbb{E}$, on which the semiflow is invertible\footnote{This follows from the fact that an amenable trajectory passing through a given point on $\mathfrak{A}$ is unique \cite{Anikushin2020Red,Anikushin2020Geom}. Note that the continuity of the inverse to $\varphi^{t}$ can be also shown without any appealing to the Brouwer theorem on invariance of domain. However, our construction of $\mathfrak{A}$ highly relies on this theorem.}. Our aim is to show that we may pass to the limit in \eqref{EQ: SqueezingDelayFlow} to get the corresponding inequality for the cocycle $(\psi,\vartheta)$ given by the linearization of $\varphi$ on $\mathfrak{A}$. Below we always assume that $F \in C^{1}(\mathbb{R}^{r};\mathbb{R}^{m})$ and $F'$ is globally bounded.

\begin{lemma}
	Let the semiflow $\varphi$ satisfy \textbf{(H1)},\textbf{(H2)} and \textbf{(H3)} as in \eqref{EQ: SqueezingDelayFlow}. Then for every $\xi \in \mathbb{E}$, $v \in \mathfrak{A}$ and $r \geq \tau$ we have
	\begin{equation}
	\label{EQ: SqueezingLinearCocycle}
	e^{2\nu r}V(L(r;v)\xi) - V(\xi) \leq -\delta \int_{0}^{r} e^{2\nu s}|L(s;v)\xi|^{2}_{\mathbb{H}}ds.
	\end{equation}
\end{lemma}
\begin{proof}
	Let $\xi \in \mathbb{E}$ and $v \in \mathfrak{A}$ be fixed. Consider \eqref{EQ: SqueezingDelayFlow} with $v_{1}:=v+h\xi$ and $v_{2}:=v$, divide both sides by $h^{2}$ and take it to the limit as $h \to 0$ to get \eqref{EQ: SqueezingLinearCocycle}. Such a passage is justified by Theorem \ref{LEM: DelayDifferentiabilityLemma}.
\end{proof}
\begin{remark}
	In fact, the restriction $r \geq \tau$ for \eqref{EQ: SqueezingLinearCocycle} is unnecessary if we use differentiability results for $\varphi^{t}$ in $\mathbb{E}$ (see Theorem 4.1, Section 2.4 in \cite{Hale1977}).
\end{remark}

Thus, the cocycle $\psi$, which is obtained through the linearization of $\varphi$ on $\mathfrak{A}$, also satisfies the hypotheses of Theorem \ref{TH: ReductionTheoremGener}, which gives us a family of amenable sets for $v \in \mathfrak{A}$, which we denote by $\mathfrak{A}'(v)$. Clearly, $\mathfrak{A}'(v)$ is a $j$-dimensional subspace of $\mathbb{E}$. Let us consider the maps $\Phi \colon \mathbb{E}^{-} \to \mathbb{E}$ and $\widetilde{\Phi} \colon \mathfrak{A} \times \mathbb{E}^{-} \to \mathbb{E}$, where $\Phi(\Pi v) = v$ for all $v \in \mathfrak{A}$ and $\widetilde{\Phi}(v, \Pi \xi) = \xi$ for all $v \in \mathfrak{A}$ and $\xi \in \mathfrak{A}'(v)$. Clearly, $\widetilde{\Phi}(v,\zeta)$ is linear in $\zeta$ so we will usually write $\widetilde{\Phi}(v)\zeta$ instead of $\widetilde{\Phi}(v,\zeta)$. One should think of $\mathfrak{A}'(v)$ as a tangent space to $\mathfrak{A}$ at $v$ and think of $\widetilde{\Phi}(v)$ as the differential of $\Phi$ at $\zeta = \Pi v$. We shall not give a proof of this, referring the interested reader to \cite{Anikushin2020Geom}. From this one can obtain the following theorem.
\begin{theorem}
	\label{TH: SmoothnessManifoldDelay}
	Let the semiflow $\varphi$ satisfy \textbf{(H1)},\textbf{(H2)} and \textbf{(H3)} as in \eqref{EQ: SqueezingDelayFlow}. Then $\Phi \colon \mathbb{E}^{-} \to \mathbb{E}$ is $C^{1}$-differentiable and $\Phi'(\zeta) = \widetilde{\Phi}(\Phi(\zeta))$ for all $\zeta \in \mathbb{E}^{-}$. Moreover, $\mathfrak{A}$ is a $C^{1}$-differentiable submanifold in $\mathbb{E}$; its tangent space at any point $v \in \mathfrak{A}$ is given by $\mathfrak{A}'(v)$; the flow $\varphi$ on $\mathfrak{A}$ is $C^{1}$-smooth and the differential of $\varphi^{t}$ at $v \in \mathfrak{A}$ is given by the map $L(t;v) \colon \mathfrak{A}'(v) \to \mathfrak{A}'(\varphi^{t}(v))$ for all $t \in \mathbb{R}$.
\end{theorem}

It is known that $C^{1}$-differentiability of $\mathfrak{A}$ is the maximum that one may achieve under the classical Spectral Gap Condition (which is included in our theory as it is shown in \cite{Anikushin2020FreqParab}) for semilinear parabolic equations and to obtain the differentiability of higher orders one needs a more restrictive condition (see S.-N.~Chow and G.R.~Sell \cite{ChowLuSell1992}; R.~Rosa and R.~Temam \cite{RosaTemam1996}). Results in the theory of normally hyperbolic manifolds (see, for example, N.~Fenichel \cite{Fenichel1971}) suggests that $\mathfrak{A}$ will be $C^{k}$-smooth if \textbf{(H3)} is satisfied for two parameters $\nu_{2} > \nu_{1}$ (with possibly different constants, operators and spaces, but with the same $j$) such that $\nu_{2}/\nu_{1} > k$. However, such large spectral gaps are rarely seen in nonlocal applications of the theory.

For the normal hyperbolicity of $\mathfrak{A}$ it is necessary to require the following property.
\begin{description}
	\item[\textbf{(HYP)}] There exists $0 < \nu' < \nu$ such that the semiflow $\varphi$ satisfies \textbf{(H1)} with a possibly different operator $P$, \textbf{(H2)} with the same $j$ and \textbf{(H3)} with $\nu$ changed to $\nu'$ and possibly different constants $\delta$ and $\tau_{V}$.
\end{description}
This property is satisfied in applications, where conditions for the existence of $P$ are given by strict inequalities so we can always vary $\nu$ a bit (see Section \ref{SEC: ExampleDelayIM}).

In \cite{Anikushin2020Geom} it is shown that the submanifold $\mathfrak{A}$ has the so-called exponential tracking property (and the exponent of attraction is nothing more than $\nu$ from \textbf{(H3)}) and it is normally hyperbolic under \textbf{(HYP)}. This, in particular, answers the questions of R.A.~Smith \cite{Smith1994} on the differentiability and normal hyperbolicity posed for reaction-diffusion equations (the geometric context is also applicable for such problems \cite{Anikushin2020FreqParab}). Let us give brief details of the constructions. Namely, for each $v_{0} \in \mathfrak{A}$ there is the \textit{stable fibre} over $v_{0}$ given by
\begin{equation}
	\mathfrak{A}^{st}(v_{0}) := \{ v \in \mathbb{E} \ | \ \int_{0}^{+\infty}e^{2\nu s}|\varphi^{s}(v)-\varphi^{s}(v_{0}) |^{2}_{\mathbb{H}} < +\infty  \}.
\end{equation}
Let $\Pi^{+} := \operatorname{Id} - \Pi$ be the complementary projector. Then one can show that $\Pi^{+} \colon \mathfrak{A}^{st}(v_{0}) \to \mathbb{E}^{+}$ is a homeomorphism and $\mathfrak{A}^{st}(v_{0})$ is a $C^{1}$-differentiable $\mathbb{E}^{+}$-submanifold in $\mathbb{E}$. These stable fibres foliate $\mathbb{E}$ and, consequently, for every $v \in \mathbb{E}$ there exists a unique $\Pi^{c}(v) \in \mathfrak{A}$ such that $v \in \mathfrak{A}^{st}(\Pi^{c}(v))$. The continuous nonlinear map $\Pi^{c}(\cdot) \colon \mathbb{E} \to \mathfrak{A}$ is called the \textit{central projector} and the triple $(\mathbb{E},\mathfrak{A},\Pi^{c})$ is a fiber bundle. Analogously, for the linearization cocycle $\psi$ one can define the stable space at $v_{0} \in \mathfrak{A}$ as
\begin{equation}
	\mathfrak{A}^{st}_{\psi}(v_{0}) := \{ \xi \in \mathbb{E} \ | \ \int_{0}^{+\infty}e^{2\nu s}| L(s;v_{0})\xi|^{2}_{\mathbb{H}} < +\infty \}.
\end{equation}
It can be verified that $\mathbb{E} = \mathfrak{A}^{st}_{\psi}(v) \oplus \mathfrak{A}'(v)$ for any $v \in \mathfrak{A}$. Such a decomposition defines a normally hyperbolic structure for $\mathfrak{A}$ with the corresponding \textit{central projectors} $\Pi^{c}_{v}$, where $\operatorname{Ran}\Pi^{c}_{v} = \mathfrak{A}'(v)$ and $\operatorname{Ker}\Pi^{c}_{v} = \mathfrak{A}^{st}_{\psi}(v)$. These constructions does not depend on the exponents from \textbf{(HYP)} and it can be shown that for some constant $M>0$ we have for all $t \geq 0$, $\xi \in \mathbb{E}$ the inequalities
\begin{equation}
	\begin{split}
		\| L(t;v) (\operatorname{Id}-\Pi^{c}_{v}) \xi \|_{\mathbb{E}} \leq M e^{-\nu t} \|\xi\|_{\mathbb{E}}, \\
		\| L(-t;v) \Pi^{c}_{v}\xi \|_{\mathbb{E}} \leq M e^{\nu' t} \| \xi \|_{\mathbb{E}}.
	\end{split}
\end{equation}
For $j=2$ these properties allow to extend results of R.A.~Smith concerned with the Poincar\'{e}-Bendixson theory \cite{Smith1992,Smith1994}, isolated periodic orbits \cite{Smith1984IsolatedOrbits} and the Poincar\'{e} index theorem \cite{Smith1984IndexTheorem}.

From the $V$-orthogonal projector $\Pi$ we obtain a chart on $\mathfrak{A}$. However, to study coordinate representation of the vector field on $\mathfrak{A}$ (i.e. inertial forms) it is convenient to use other charts on $\mathfrak{A}$. Namely, under \textbf{(H3)} we call a bounded projector $\Pi^{a}$ in $\mathbb{E}$ \textit{admissible} if for any nonzero $v \in \operatorname{Ran}\Pi^{a}$ we have $V(v) < 0$ and for any nonzero $v \in \operatorname{Ker}\Pi^{a}$ we have $V(v) > 0$. Under the conditions of Theorem \ref{TH: SmoothnessManifoldDelay} it turns out that $\Pi^{a} \colon \mathfrak{A} \to \operatorname{Ran}\Pi^{a}$ is also a homeomorphism and $\Pi^{a} \colon \mathfrak{A}'(v) \to \operatorname{Ran}\Pi^{a}$ is an isomorphism for any $v \in \mathfrak{A}$. Consequently, due to the $C^{1}$-differentiability of $\mathfrak{A}$ and linearity of $\Pi^{a}$, the inverse map $\Phi^{a}$ is also $C^{1}$-differentiable and globally Lipschitz. Moreover, transition maps between such admissible charts are globally Lipschitz and $C^{1}$-differentiable (see \cite{Anikushin2020Geom}).

The following lemma describes the dynamics of $\varphi$ on $\mathfrak{A}$ and its linearization $L(t;v)$ on $\mathfrak{A}'(v)$ by ordinary differential equations in $\mathbb{E}^{-}$. Equation \eqref{EQ: ProjectedNonlinearDelay} is called the \textit{inertial form} of $\varphi$ on $\mathfrak{A}$.  We give a sketch of its proof based on the above introduced arguments.
\begin{lemma}
	\label{LEM: ProjectedEquationsDelay}
	Let the semiflow $\varphi$ satisfy \textbf{(H1)},\textbf{(H2)} and \textbf{(H3)} as in \eqref{EQ: SqueezingDelayFlow}. Suppose there exists a spectral projector for $A$ which is admissible. Then
	
	1. There is a one-to-one correspondence between the trajectories of the flow $\varphi$ on $\mathfrak{A}$ and solutions of the following ODE in $\mathbb{E}^{-}$:
	\begin{equation}
	\label{EQ: ProjectedNonlinearDelay}
	\dot{\zeta}(t) = \Pi \left[ A \Phi(\zeta(t)) + BF(C\Phi(\zeta(t))) \right] =: f(\zeta(t)),
	\end{equation}
	where the vector field $f \colon \mathbb{E}^{-} \to \mathbb{E}^{-}$ is $C^{1}$-differentiable. This correspondence is given by the identities $\Phi(\zeta(t;\zeta_{0}))=\varphi^{t}(\Phi(\zeta_{0}))$ and $\Pi \varphi^{t}(\Phi(\zeta_{0})) = \zeta(t;\zeta_{0})$ for all $t \in \mathbb{R}$ and $\zeta_{0} \in \mathbb{E}^{-}$.
	
	2. There is a one-to-one correspondence between the trajectories of the cocycle $\psi$ and solutions of the following ODE in $\mathbb{E}^{-}$:
	\begin{equation}
	\label{EQ: ProjectedLinearDelay}
	\dot{\eta}(t) = \Pi [A + B F'(C\varphi^{t}(v))C] \widetilde{\Phi}(\varphi^{t}(v))\eta(t) =: A_{L}(\varphi^{t}(v)) \eta(t),
	\end{equation}
	where the linear operator $A_{L}(v) \colon \mathbb{E}^{-} \to \mathbb{E}^{-}$ depend continuously on $v \in \mathfrak{A}$ and norms of $A_{L}(v)$ are uniformly bounded in $v \in \mathfrak{A}$. This correspondence is given by the identities $\widetilde{\Phi}(\varphi^{t}(v))\eta(t;\eta_{0};v) = L(t;v) \widetilde{\Phi}(v)\eta_{0}$ and $\Pi L(t;v) \widetilde{\Phi}(v)\eta_{0} = \eta(t;\eta_{0};v) $ for all $t \in \mathbb{R}$, $v \in \mathfrak{A}$ and $\eta_{0} \in \mathbb{E}^{-}$.
\end{lemma}
\begin{proof}
	The main obstacle to study the vector field $f$ in \eqref{EQ: ProjectedNonlinearDelay} is due to fact that the $V$-orthogonal projector $\Pi$ has nothing to do with the operators $A$ and $B$ and, consequently, the differentiability properties of $f$ are not obvious in the chart given by $\Pi$. Due to this, it is convenient to use different charts on $\mathfrak{A}$ given by other projectors.
	
	By the hypothesis, there exists a spectral projector $\Pi^{a}$ for $A$ which is admissible. Note that such a projector can be defined everywhere in $\mathbb{H}$ (since $A$ is an operator in $\mathbb{H}$) and commutes with $A$. Thus, $\Pi^{a} A = A \Pi^{a}$ and $\Pi^{a} B$ are bounded linear maps and in the chart given by $\Pi^{a}$ the conclusions of Lemma \ref{LEM: ProjectedEquationsDelay} are obvious. Then the arguments before lemma show that the same holds in any admissible chart and, in particular, the one given by the $V$-orthogonal projector $\Pi$.
\end{proof}

The following theorem is a generalization of Corollary 2.2 from \cite{Smith1986HD}, where the case of ODEs in $\mathbb{R}^{n}$ with $j=n$ is considered. The main idea is to use the quadratic form $-V(\cdot)$ restricted to each tangent space $\mathfrak{A}'(v)$ as a Riemannian metric on $\mathfrak{A}$. Condition \eqref{EQ: SqueezingLinearCocycle} gives a lower bound for the singular values of $L(t;v) \colon \mathfrak{A}'(v) \to \mathfrak{A}'(\varphi^{t}(v))$ in this metric. This lower bound can be used to estimate a product of the first $l < j$ singular values through the product of all $j$ singular values, which, in turn, can be estimated through the trace that is independent of metric changes.
\begin{theorem}
	\label{TH: FrequencyDimEstimateDelay}
	Under the conditions of Lemma \ref{LEM: ProjectedEquationsDelay} suppose that $\mathcal{K}$ is an invariant compact and for some $d \in [0,j]$ we have
	\begin{equation}
	\label{EQ: DimFreqEstimateMainCond}
	(j-d)\nu + \operatorname{Tr}(A_{L}(v)) < 0 \text{ for all } v \in \mathcal{K},
	\end{equation}
	where $A_{L}(\cdot)$ is defined in \eqref{EQ: ProjectedLinearDelay}. Then $\operatorname{dim}_{\operatorname{F}}\mathcal{K} < d$.
\end{theorem}
\begin{proof}
	It is sufficient to estimate the fractal dimension of $\Pi \mathcal{K}$, which is an invariant compact w.~r.~t. the flow given by \eqref{EQ: ProjectedNonlinearDelay} in $\mathbb{E}^{-}$. From Theorem \ref{TH: SmoothnessManifoldDelay} and Lemma \ref{LEM: ProjectedEquationsDelay} it follows that the vector field $f \colon \mathbb{E}^{-} \to \mathbb{E}^{-}$ from \eqref{EQ: ProjectedNonlinearDelay} is $C^{1}$-smooth and its derivative at $\zeta \in \mathbb{E}^{-}$ is given by $A_{L}(\Phi(\zeta))$. For any $t \geq 0$ and $\zeta_{0} \in \mathbb{E}^{-}$ put $\widetilde{L}(t;\zeta_{0}) := \Pi L(t;\Phi(\zeta_{0}))\Phi'(\zeta_{0})$. Clearly, $\widetilde{L}(t;\zeta_{0}) \colon \mathbb{E}^{-} \to \mathbb{E}^{-}$ is the differential at $\zeta_{0}$ of the map $\mathbb{E}^{-} \ni \zeta_{0} \mapsto \zeta(t;\zeta_{0}) \in \mathbb{E}^{-}$. Let us endow the space $\mathbb{E}^{-}$ with any scalar product. Let $\sigma_{1}(t;\zeta_{0}) \geq \ldots \geq \sigma_{j}(t;\zeta_{0})$ denote the singular values of $\widetilde{L}(t;\zeta_{0})$ w.~r.~t. this scalar product. From the Liouville trace formula for \eqref{EQ: ProjectedLinearDelay}, we have
	\begin{equation}
	\label{EQ: FreqDimEstTraceFormula}
	\sigma_{1}(t;\zeta_{0}) \cdot \ldots \cdot \sigma_{j}(t;\zeta_{0}) \leq \operatorname{exp}\left(\int_{0}^{t} \operatorname{Tr}A_{L}(\varphi^{s}(\Phi(\zeta_{0})))ds \right).
	\end{equation}
	Now for a fixed $t \geq 0$ and $\zeta_{0} \in \mathbb{E}^{-}$ let us consider the scalar products in $\mathbb{E}^{-}$ defined trough the quadratic forms $( \eta, \eta )_{1} := -V(\Phi'(\zeta_{0})\eta)$ and $(\eta,\eta)_{2}:= - V( \Phi'(\zeta(t;\zeta_{0}))\eta )$ for all $\eta \in \mathbb{E}^{-}$. Let $\mathbb{E}_{1}$ and $\mathbb{E}_{2}$ be the space $\mathbb{E}^{-}$ endowed with scalar products $(\cdot,\cdot)_{1}$ and $(\cdot,\cdot)_{2}$ respectively. Let $\sigma'_{1}(t;\zeta_{0}) \geq \ldots \geq \sigma'_{j}(t;\zeta_{0})$ denote the singular values of $\widetilde{L}(t;\zeta_{0}) \colon \mathbb{E}_{1} \to \mathbb{E}_{2}$. By the Courant–Fischer–Weyl min-max principle, we have for $k=1,\ldots,j$ that
	\begin{equation}
	\label{EQ: FreqEstimateDelayMinMax}
	\left(\sigma'_{k}(t;\zeta_{0})\right)^{2} = \sup_{ \mathbb{W}_{k} \subset \mathbb{E}^{-} } \inf_{ \eta \in \mathbb{W}_{k}, \eta \not=0} \frac{(\widetilde{L}(t;\zeta_{0})\eta, \widetilde{L}(t;\zeta_{0})\eta)_{2}}{(\eta,\eta)_{1}},
	\end{equation}
	where the supremum is taken over all $k$-dimensional subspaces $\mathbb{W}_{k}$ of $\mathbb{E}^{-}$. From \eqref{EQ: SqueezingLinearCocycle} with $\xi:=\Phi'(\zeta_{0})\eta$, $r:=t$ and $v:=\Phi(\zeta_{0})$ we have
	\begin{equation}
	\frac{(\widetilde{L}(t;\zeta_{0})\eta, \widetilde{L}(t;\zeta_{0})\eta)_{2}}{(\eta,\eta)_{1}} \geq e^{-2\nu t}.
	\end{equation}
	From this and \eqref{EQ: FreqEstimateDelayMinMax} we have $\sigma'_{k}(t;\zeta_{0}) \geq e^{- \nu t}$. Moreover, from \eqref{EQ: FreqEstimateDelayMinMax} and the min-max principle for $\sigma_{k}$ it is also clear that there exists a constant $C=C(t;\zeta_{0}) > 0$ such that $\sigma_{k}(t;\zeta_{0}) \geq C(t;\zeta_{0}) \sigma'(t;\zeta_{0})$ for all $t \geq 0$ and $\zeta_{0} \in \mathbb{E}^{-}$. Moreover, since $\mathcal{K}$ is compact and $\Phi'(\zeta_{0})$ depend continuously on $\zeta_{0}$, there exists a constant $C_{\mathcal{K}}>0$ such that $\sigma_{k}(t;\zeta_{0}) \geq C_{\mathcal{K}} \sigma'_{k}(t;\zeta_{0}) \geq C_{\mathcal{K}} e^{-\nu t}$ for all $t \geq 0$ and $\zeta_{0} \in \Pi\mathcal{K}$. From this and \eqref{EQ: FreqDimEstTraceFormula} for any $k=1,\ldots,j$ we have
	\begin{equation}
	\label{EQ: DimFreqEstIntegerFormula}
	\begin{split}
	\sigma_{1}(t;\zeta_{0}) \cdot \ldots \sigma_{k}(t;\zeta_{0}) = \frac{\sigma_{1}(t;\zeta_{0}) \cdot \ldots \cdot \sigma_{j}(t;\zeta_{0})}{\sigma_{k+1}(t;\zeta_{0}) \cdot \ldots \cdot \sigma_{j}(t;\zeta_{0})} \leq \\ \leq C^{-(j-k)}_{K} \operatorname{exp}\left( \int_{0}^{t} \left[(j-k)\nu + \operatorname{Tr}(A_{L}(\varphi^{s}(\Phi(\zeta_{0})))) \right]ds \right).
	\end{split}
	\end{equation}
	Now let $d = k + \delta$, where $k$ is a nonnegative integer and $\delta \in (0,1]$. Since
	\begin{equation}
	\begin{split}
	\sigma_{1}(t;\zeta_{0}) \cdot \ldots \cdot \sigma_{k}(t;\zeta_{0}) \cdot \sigma_{k+1}^{\delta} =\\= \left(\sigma_{1}(t;\zeta_{0}) \cdot \ldots \cdot \sigma_{k}(t;\zeta_{0}) \right)^{1-\delta} \cdot \left(\sigma_{1}(t;\zeta_{0}) \cdot \ldots \cdot \sigma_{k+1}(t;\zeta_{0}) \right)^{\delta},
	\end{split}	
	\end{equation} 
    from \eqref{EQ: DimFreqEstIntegerFormula} we have
    \begin{equation}
    \label{EQ: DimFreqEstEstimateForLimit}
    \begin{split}
    \sigma_{1}(t;\zeta_{0}) \cdot \ldots \cdot \sigma_{k}(t;\zeta_{0}) \cdot \left(\sigma_{k+1}(t;\zeta_{0})\right)^{\delta} \leq \\ \leq C^{-(j-d)}_{K} \operatorname{exp}\left( \int_{0}^{t} \left[(j-d)\nu + \operatorname{Tr}(A_{L}(\varphi^{s}(\Phi(\zeta_{0})))) \right]ds \right).
    \end{split}
    \end{equation}
    If \eqref{EQ: DimFreqEstimateMainCond} is satisfied then the right-hand side of \eqref{EQ: DimFreqEstEstimateForLimit} tends to $0$ as $t \to +\infty$. But this implies that $\operatorname{dim}_{\operatorname{F}}\Pi \mathcal{K} < d$ (see Theorem 2.1 from \cite{ChepyzhovIlyin2004}) and since $\mathcal{K} = \Phi(\Pi \mathcal{K})$ and $\Phi$ is a Lipschitz map, we get $\operatorname{dim}_{\operatorname{F}}\mathcal{K} < d$.
\end{proof}
\section{Example (continued)}
\label{SEC: ExampleDelayIM}

Let us consider again the Suarez-Schopf model \eqref{EQ: ElNinoSSmodel}. Here we will provide conditions for the existence of one-dimensional and two-dimensional inertial manifolds in the model. These manifolds are actually perturbations of the eigenspaces corresponding to the leading characteristic roots $\lambda_{1}$ and $\lambda_{2}$ at the zero equilibrium $\phi^{0}$. Conditions for persistence of such eigenspaces are given in terms of the so-called transfer function of the linear part (see \cite{Anikushin2020Freq}) and take the form of an inequality (called frequency inequality), which generalizes the Spectral Gap Condition \cite{Anikushin2020FreqParab}. In the case of \eqref{EQ: ElNinoSSmodel} written in the form \eqref{EQ: AbstractDelayHilberSpace} the transfer function $W(p)=C(A-pI)^{-1}B$ is given by
\begin{equation}
W(p) = \frac{-1}{1-\alpha e^{-\tau p} - p}.
\end{equation}
Here we assumed that $\widetilde{A}\phi = \phi(0) - \alpha \phi(-\tau)$, $\widetilde{B}=-1$, $\widetilde{C}\phi = \phi(0)$ and $F(y) = y^{3}$.

Let $\Lambda$ be the Lipschitz constant of $x^{3}$ on $[-R_{1},R_{1}]$, where $R_{1}$ is defined in Lemma \ref{LEM: SharperRegionSS}, that is $\Lambda = 3 R^{2}_{1}$.
\begin{theorem}
	\label{TH: SSmodelOneDimIMSharper}
	Suppose that $\alpha \tau < 1/2$ and for some $\nu \in (0,-\lambda_{2})$ the frequency inequality
	\begin{equation}
		\label{EQ: FrequencyIneqSSmodelOnedim}
		\operatorname{Re}W(-\nu + i\omega) + \Lambda^{-1} > 0 \text{ for all } \omega \in \mathbb{R}
	\end{equation}
    is satisfied. Then the global attractor $\mathcal{K}$ of \eqref{EQ: ElNinoSSmodel} is contained in a one-dimensional $C^{1}$-differentiable inertial manifold.
\end{theorem}
The proof is standard: we truncate\footnote{Note that it is important to preserve the condition $g'(x) \geq 0$.} the nonlinearity $x^{3}$ outside a small neighborhood of $\mathcal{K}$ and study the modified equation. From \eqref{EQ: FrequencyIneqSSmodelOnedim} we can apply\footnote{For this we consider the quadratic form (in terms of \cite{Anikushin2020Freq}) given by $\mathcal{F}(v,\xi) = \xi \cdot (\Lambda Cv - \xi)$, where $\xi \in \Xi := \mathbb{R}$, $v=(\phi(0),\phi)$ and $Cv = \phi(0)$.} the Frequency Theorem from \cite{Anikushin2020Freq} to get a proper operator $P$ as in Section \ref{SEC: InvariantManifoldsDelay}. Then we can construct an inertial manifold for the modified system, which contains $\mathcal{K}$ by definition (since $\nu >0$, see Appendix \eqref{APP: AdditionToReduction}) by methods of \cite{Anikushin2020Geom}. Restricting the manifold to the small neighborhood of $\mathcal{K}$ gives the desired manifold.

Thus, under the conditions of Theorem \ref{TH: SSmodelOneDimIMSharper}, $\mathcal{K}$ is just the union of equilibria and the unstable manifold of $\phi^{0}$.

Numerical simulations in \cite{Anikushin2021SS} show that the global attractor $\mathcal{K}$ for a wide range of parameters of interest lies in the ball of radius $1$. Let $\lambda_{3}$ and $\lambda_{4}$ be the first pair of complex-conjugate roots of \eqref{EQ: RootsSuarezDelay}.
\begin{theorem}
	\label{TH: SSmodelTwoDimIMSharper}
	Suppose for some $\nu \in (-\lambda_{2},-\operatorname{Re}\lambda_{3})$ the frequency inequality
	\begin{equation}
		\label{EQ: FrequencyIneqSSmodelTwodim}
		\operatorname{Re}W(-\nu + i\omega) + 1/3 > 0 \text{ for all } \omega \in \mathbb{R}
	\end{equation}
	is satisfied. Then any invariant set of \eqref{EQ: ElNinoSSmodel}, lying in the ball of radius $1$, is contained in a two-dimensional $C^{1}$-differentiable inertial manifold.
\end{theorem}

Fig. \ref{FIG: InertialManifoldsSSmodelSharper} shows a region where the conditions of Theorem \ref{TH: SSmodelOneDimIMSharper} or Theorem \ref{TH: SSmodelTwoDimIMSharper} are satisfied. Note that a part of the blue region can be obtained also using the rough estimate $\sqrt{1+\alpha}$ from Lemma \ref{LEM: RegionEstimateSS} to bound the global attractor $\mathcal{K}$. This allows to rigorously justify the existence of two-dimensional inertial manifolds containing $\mathcal{K}$. However, this part does not contain interesting parameters from \cite{Anikushin2021SS}, for which periodic orbits can be observed. We expect that the relative simplicity of the Suarez-Schopf model may allow to construct sharper regions and, consequently, rigorously justify the existence of inertial manifolds for the interesting parameters.

\begin{figure}
	\centering
	\includegraphics[width=1.\linewidth]{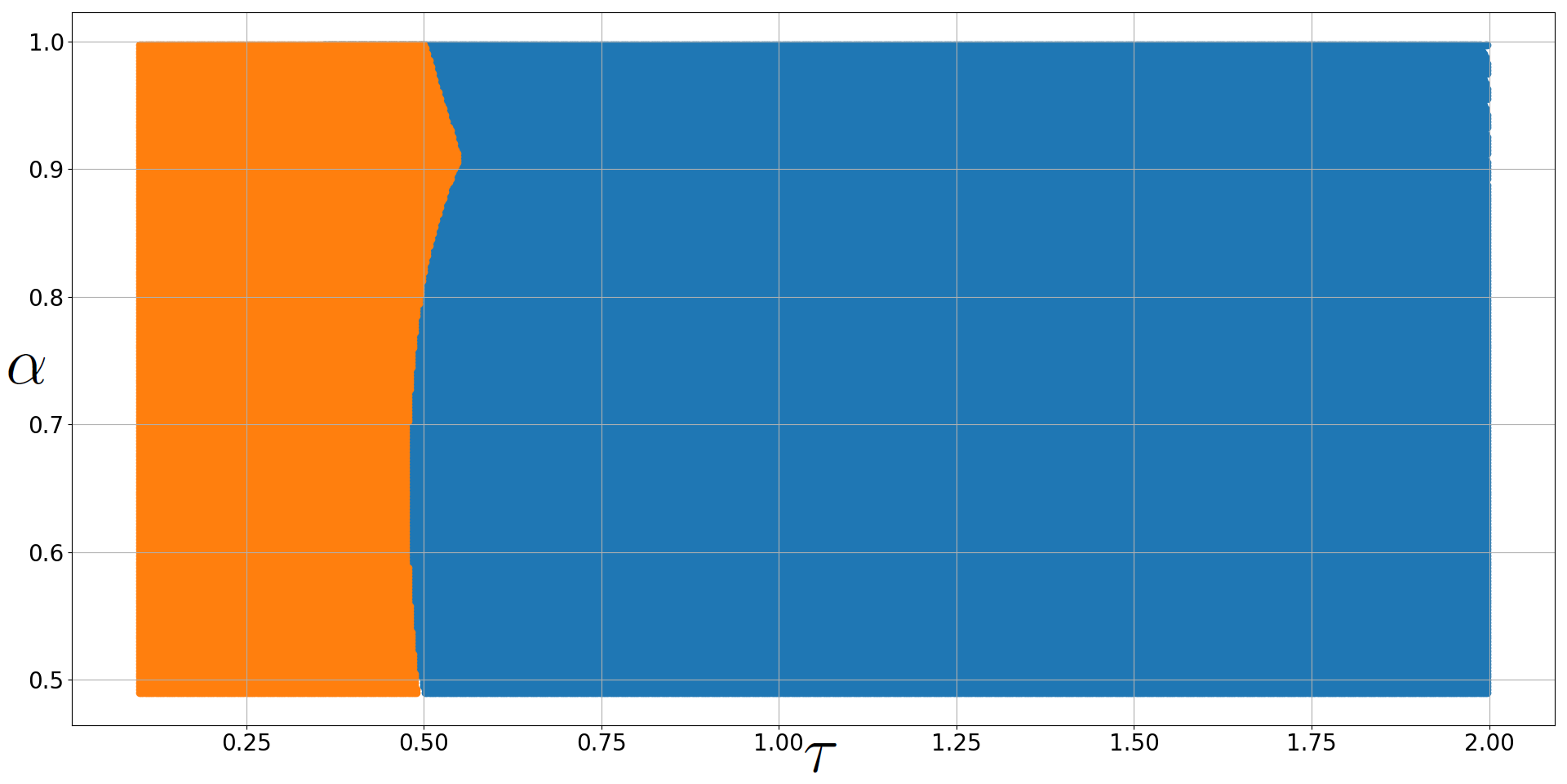}	
	\caption{A numerically obtained region in the space of parameters $(\tau,\alpha)$ of system \eqref{EQ: ElNinoSSmodel} for which the conditions of Theorem \ref{TH: SSmodelOneDimIMSharper} (orange) or Theorem \ref{TH: SSmodelTwoDimIMSharper} (blue) are satisfied.}
	\label{FIG: InertialManifoldsSSmodelSharper}
\end{figure}
\section*{Funding}
The reported study was funded by the Russian Science Foundation (Project 22-11-00172).
\appendix
\section{An addition to the reduction principle}
\label{APP: AdditionToReduction}
Let $\mathcal{Q}$ be a metric space and $\vartheta^{t} \colon \mathcal{Q} \to \mathcal{Q}$, where $t \in \mathbb{R}$, be a flow on $\mathcal{Q}$. Let $\mathbb{E}$ be a real Banach space. We say that the family of maps $\psi^{t}(q,\cdot) \colon \mathbb{E} \to \mathbb{E}$, where $q \in \mathcal{Q}$ and $t \geq 0$, is a \textit{cocycle} in $\mathbb{E}$ over the \textit{driving system} $(\mathcal{Q},\vartheta)$ if
\begin{enumerate}
	\item[1.] $\psi^{0}(q,v) = v$ and $\psi^{t+s}(q,v) = \psi^{t}(\vartheta^{s}(q),\psi^{s}(q,v))$ for all $t, s \geq 0$, $v \in \mathbb{E}$ and $q \in \mathcal{Q}$.
	\item[2.] The map $(t,q,v) \mapsto \psi^{t}(q,v)$ is continuous as a map $\mathbb{R}_{+} \times \mathcal{Q} \times \mathbb{E} \to \mathbb{E}$.
\end{enumerate}

Now suppose that $\mathbb{E}$ is continuously embedded into some real normed space $\mathbb{H}$, which is assumed to be Hilbert for convenience We identify elements of $\mathbb{E}$ and $\mathbb{H}$ under such an embedding. Let $\langle v, f \rangle$ denote the dual pairing between $v \in \mathbb{E}$ and $f \in \mathbb{E}^{*}$.

Let $P \colon \mathbb{E} \to \mathbb{E}^{*}$ be a bounded linear operator, i.~e. $P \in \mathcal{L}(\mathbb{E};\mathbb{E}^{*})$. We say that $P$ is \textit{symmetric} if $\langle v_{1}, Pv_{2} \rangle = \langle v_{2}, Pv_{1} \rangle$ for all $v_{1},v_{2} \in \mathbb{E}$. For a subspace $\mathbb{L} \subset \mathbb{E}$ we say that $P$ is \textit{positive} (respectively, \textit{negative}) on $\mathbb{L}$  if $\langle v, Pv \rangle > 0$ (respectively, $\langle v, Pv \rangle < 0$) for all $v \in \mathbb{L}$ with $v \not= 0$.

\begin{description}
	\item[\textbf{(H1)}] There is $P \in \mathcal{L}(\mathbb{E};\mathbb{E}^{*})$, which is symmetric and such that $\mathbb{E}$ splits into the direct sum of some subspaces $\mathbb{E}^{+}$ and $\mathbb{E}^{-}$, i.~e. $\mathbb{E} = \mathbb{E}^{+} \oplus \mathbb{E}^{-}$, such that  $P$ is positive on $\mathbb{E}^{+}$ and negative on $\mathbb{E}^{-}$.
	\item[\textbf{(H2)}] For some integer $j \geq 0$ we have $\dim \mathbb{E}^{-} = j$.
	\item[\textbf{(H3)}] For $V(v):=\langle v, Pv \rangle$ and some numbers $\delta>0$, $\nu>0$, $\tau_{V} \geq 0$ we have
	\begin{multline}
	\label{EQ: SqueezingProperty}
	e^{2\nu r}V(\psi^{r}(q,v_{1})-\psi^{r}(q,v_{2}))-e^{2\nu l} V(\psi^{l}(q,v_{1})-\psi^{l}(q,v_{2})) \leq \\ \leq -\delta\int_{l}^{r} e^{2\nu s} |\psi^{s}(q,v_{1})-\psi^{s}(q,v_{2})|^{2}_{\mathbb{H}}ds
	\end{multline}
	satisfied for every $v_{1},v_{2} \in \mathbb{E}$, $q \in \mathcal{Q}$ and $0 \leq l \leq r$ such that $r - l \geq \tau_{V}$.
\end{description}

Note that due to the cocycle property it is sufficient to require that \textbf{(H3)} holds for $l=0$ and any $r \geq \tau_{V}$. A more general case when $P$ and $\nu$ depends on $q \in \mathcal{Q}$ is considered in \cite{Anikushin2020Geom}.

\begin{remark}
	\label{REM: VorthProjector}
	Let the assumptions \textbf{(H1)} and \textbf{(H2)} be satisfied. Since $\mathbb{E}^{-}$ is finite-dimensional and $V(\cdot)$ is of constant sign on it, one can find the $V$-orthogonal complement of $\mathbb{E}^{-}$ in $\mathbb{E}$, which is defined as follows. The symmetric bilinear form $\langle v_{1}, -Pv_{2} \rangle$ defines an inner product in $\mathbb{E}^{-}$. Every $v \in \mathbb{E}$ gives rise to a continuous linear functional on $\mathbb{E}^{-}$ as $\langle \cdot, Pv \rangle$. By the Riesz representation theorem, there exists a unique element $\Pi v \in \mathbb{E}^{-}$ such that $\langle w, P v \rangle = \langle w,  P\Pi v \rangle$ for all $w \in \mathbb{E}^{-}$. Clearly, $\Pi \in \mathcal{L}(\mathbb{E};\mathbb{E})$. Put $\mathbb{E}^{\bot,V} := \operatorname{Ker}(I-\Pi)$. It is easy to verify that $\mathbb{E} = \mathbb{E}^{\bot,V} \oplus \mathbb{E}^{-}$ and $P$ is positive on $\mathbb{E}^{\bot,V}$. Thus, we can always assume that the subspaces $\mathbb{E}^{+}$ and $\mathbb{E}^{-}$ from $\textbf{(H1)}$ are $V$-orthogonal in the sense that $V(v)=V(v^{+})+V(v^{-})$, where $v=v^{+}+v^{-}$ is the unique decomposition with $v^{+} \in \mathbb{E}^{+}$ and $v^{-} \in \mathbb{E}^{-}$. We say that $\Pi$ is the $V$-\textit{orthogonal projector} onto $\mathbb{E}^{-}$. For $\mathbb{E} = \mathbb{H}$ this construction of $V$-orthogonal projectors also allows to avoid the use of the continuous functional calculus for bounded self-adjoint operators considered in \cite{Anikushin2020Red}.
\end{remark}

Let $v(\cdot) \colon \mathbb{R} \to \mathbb{E}$ be a continuous function. We say that $v(\cdot)$ is a \textit{complete trajectory} of the cocycle if there exists $q \in \mathcal{Q}$ such that $v(t+s) = \psi^{t}(\vartheta^{s}(q),v(s))$ for all $t \geq 0$ and $s \in \mathbb{R}$. In this case we say that $v(\cdot)$ is passing through $v(0)$ at $q$ (or simple, $v(\cdot)$ is a complete trajectory at $q$). Under \textbf{(H3)} a complete trajectory $v(\cdot)$ is called \textit{amenable} if
\begin{equation}
\int_{-\infty}^{0}e^{2\nu s} |v(s)|^{2}_{\mathbb{H}} ds < +\infty.
\end{equation}

Let us consider the following assumption.
\begin{description}
	\item[\textbf{(S)}] There is a number $t_{S} \geq 0$ and there is a constant $C_{S}>0$ such that
	\begin{equation}
	\|\psi^{t_{emb}}(q,v_{1})-\psi^{t_{emb}}(q,v_{2})\|_{\mathbb{E}} \leq C_{S} |v_{1}-v_{2}|_{\mathbb{H}} \text{ for all } q \in \mathcal{Q} \text{ and } v_{1},v_{2} \in \mathbb{E}.
	\end{equation}
\end{description}
Note that in the case $\mathbb{E} = \mathbb{H}$ considered in \cite{Anikushin2020Red} we have \textbf{(S)} automatically satisfied with $t_{S} = 0$ and $C_{S} = 1$.

In the construction of invariant manifolds by the methods of \cite{Anikushin2020Red} a compactness assumption is required as follows.
\begin{description}
	\item[\textbf{(COM)}] There exists $t_{com} > 0$ such that the map $\psi^{t_{com}}(q,\cdot) \colon \mathbb{E} \to \mathbb{E}$ is compact for all $q \in \mathcal{Q}$.
\end{description}

Let $\mathfrak{A}(q)$ be the set of all amenable trajectories at $q$. A generalization of Theorem 1 from \cite{Anikushin2020Red} can be given as follows.
\begin{theorem}
	\label{TH: ReductionTheoremGener}
	Let the cocycle $\psi$ in $\mathbb{E}$ satisfy \textbf{(H1)},\textbf{(H2)}, \textbf{(H3)}, \textbf{(S)} and \textbf{(COM)}. Let $\Pi$ be the $V$-orthogonal projector onto $\mathbb{E}^{-}$ (see Remark \ref{REM: VorthProjector}). Then for any $q \in \mathcal{Q}$ either $\mathfrak{A}(q)$ is empty or the map $\Pi_{q} :=\restr{\Pi}{\mathfrak{A}(q)} \colon \mathfrak{A}(q) \to \mathbb{E}^{-}$ is a homeomorphism.
\end{theorem}
A proof can be given following the same arguments as in \cite{Anikushin2020Red} and using the above given remarks. A proof for the more general case, where $P$ and $\nu$ depends on $q \in \mathcal{Q}$, is presented in our adjacent work \cite{Anikushin2020Geom}. Note that we also considered a number $\tau_{V} \geq 0$ in \textbf{(H1)} (in \cite{Anikushin2020Red} there was no such a number), which may be convenient sometimes. Convergence theorems for periodic cocycles given in \cite{Anikushin2020Red} can be also generalized to this setting (see \cite{Anikushin2020Geom}). Along with the Frequency Theorem from \cite{Anikushin2020Freq}, which covers delay equations and allows to construct operators as in \textbf{(H3)}, this satisfactorily solves Problem 2 from \cite{Anikushin2020Red}.

Let the conditions of Theorem \ref{TH: ReductionTheoremGener} be satisfied and the sets $\mathfrak{A}(q)$ be not empty for all $q \in \mathcal{Q}$. Then we can consider the map $\Phi \colon \mathcal{Q} \times \mathbb{E}^{-} \to \mathfrak{A}(q)$ defined as $\Phi(q,\zeta) := \Phi_{q}(\zeta)$, where $\Phi_{q} \colon \mathbb{E}^{-} \to \mathfrak{A}(q)$ is the inverse to $\Pi_{q}=\restr{\Pi}{\mathfrak{A}(q)} \colon \mathfrak{A}(q) \to \mathbb{E}^{-}$. For every fixed $q$ the map $\Phi(q,\cdot)$ is a homeomorphism between $\mathbb{E}^{-}$ and $\mathfrak{A}(q)$ due to Theorem \ref{TH: ReductionTheoremGener}. In \cite{Anikushin2020Red} it was posed a problem (see Problem 1 therein), whether the map $\Phi \colon \mathcal{Q} \times \mathbb{E}^{-} \to \mathbb{E}$ is continuous. Here we present an answer, which extends the case of semiflows and periodic cocycles studied in \cite{Anikushin2020Red} and covers many interesting cases arising in practice (see Remark \ref{REM: ApplicationsofContTh} below).

\begin{description}
	\item[\textbf{(BA)}] For any $q \in \mathcal{Q}$ there is a bounded in the past complete trajectory $w^{*}_{q}(\cdot)$ at $q$ and there exists a constant $M_{b}>0$ such that $\sup_{t \leq 0}\|w^{*}_{q}(t)\|_{\mathbb{E}} \leq M_{b}$ for all $q \in \mathcal{Q}$.
\end{description}

We say that $\Phi \colon \mathcal{Q} \times \mathbb{E}^{-} \to \mathbb{E}$ is uniformly compact if $\Phi(\mathcal{C},\mathcal{B})$ is precompact in $\mathbb{E}$ for any precompact set $\mathcal{C} \subset \mathcal{Q}$ and bounded set $\mathcal{B} \subset \mathbb{E}^{-}$.
\begin{theorem}
	\label{TH: PhiContinuityGen}
	Under the hypotheses of Theorem \ref{TH: ReductionTheoremGener} suppose in addition that \textbf{(BA)} holds and the map $\Phi$ is uniformly compact. Then $\Phi \colon \mathcal{Q} \times \mathbb{E}^{-} \to \mathbb{E}$ is continuous.
\end{theorem}
\begin{proof}
	Let $q_{k}, q \in \mathcal{Q}$ and $\zeta_{k},\zeta \in \mathbb{E}^{-}$, where $k=1,2,\ldots$, and $q_{k} \to q$, $\zeta_{k} \to \zeta$ as $k \to +\infty$. Let $v_{k}(\cdot), v(\cdot)$ be amenable trajectories at $q_{k}$ and $q$ respectively and such that $\Phi(q_{k},\zeta_{k}) = v_{k}(0)$ and $\Phi(q,\zeta) = v(0)$. Then from \textbf{(H3)} for $l \leq -\tau_{V}$ we have
	\begin{equation}
	\label{EQ: TheoremPhiContH3First}
	V(v_{k}(0)-w_{q_{k}}(0)) - e^{2\nu l} V(v_{k}(l)-w_{q_{k}}(l)) \leq -\delta \int_{l}^{0} e^{2\nu s}|v_{k}(s)-w_{q_{k}}(s)|^{2}_{\mathbb{H}}ds.
	\end{equation}
	If $k$ is fixed, a subsequence $l_{m}$, $m=1,2,\ldots$,  $l_{m} \to -\infty$ as $m \to +\infty$ can be chosen\footnote{Since $v_{k}(\cdot)$ and $w_{q_{k}}(\cdot)$ are amenable, we have $\int_{-\infty}^{0}e^{2\nu s}\| v_{k}(s)-w_{q_{k}}(s) \|^{2}_{\mathbb{E}}ds < +\infty$. To get $l_{m}$ one may apply the mean value formula to the previous integral on $[l-1,l]$ for $l=-1,2,\ldots$.} such that $e^{2\nu l_{m}} V(v_{k}(l_{m})-w_{q_{k}}(l_{m})) \to 0$ as $m \to +\infty$. Putting $l=l_{m}$ in \eqref{EQ: TheoremPhiContH3First} and taking it to the limit as $m \to +\infty$, we get
	\begin{equation}
	\label{EQ: ContIntegalDiffEstimate}
	-\delta^{-1} V(v_{k}(0)-w_{q_{k}}(0)) \geq \int_{-\infty}^{0}e^{2 \nu s} |v_{k}(s)-w_{q_{k}}(s)|^{2}_{\mathbb{H}}ds 
	\end{equation}
	Since $\mathbb{E}^{+}$ and $\mathbb{E}^{-}$ are $V$-orthogonal, we have $-V(v_{k}(0)-w_{q_{k}}(0)) \leq \|\Pi \| \cdot \|\zeta_{k}-\Pi w_{q_{k}}(0)\|^{2}_{\mathbb{E}}$. Using the uniform boundedness of $w_{q_{k}}$ we have
	\begin{equation}
	\begin{split}
	\left(\int_{-\infty}^{0}e^{2\nu s}|v_{k}(s)-w_{q_{k}}(s)|^{2}_{\mathbb{H}}ds \right)^{1/2} \geq \left(\int_{-\infty}^{0}e^{2\nu s}|v_{k}(s)|^{2}_{\mathbb{H}}ds \right)^{1/2} -\\- \left(\int_{-\infty}^{0}e^{2\nu s}|w_{q_{k}}(s)|^{2}_{\mathbb{H}}ds \right)^{1/2} \geq \left(\int_{-\infty}^{0}e^{2\nu s}|v_{k}(s)|^{2}_{\mathbb{H}}ds \right)^{1/2} - \frac{M}{2\nu}.
	\end{split}
	\end{equation}
	From this, \eqref{EQ: ContIntegalDiffEstimate} and \textbf{(S)} there exists a constant $\widetilde{M}$ (independent of $k$) such that
	\begin{equation}
	\label{EQ: ContinuityThBoundedIntegral}
	\widetilde{M} \geq \int_{-\infty}^{0}e^{2\nu s}\| v_{k}(s) \|^{2}_{\mathbb{E}}ds.
	\end{equation}
	Now suppose that $\Phi(q_{k},\zeta_{k})=v_{k}(0)$ does not converge in $\mathbb{E}$ to $v(0)=\Phi(q,\zeta)$. Then for a subsequence (we keep the same index) we have that $|v_{k}(0)-v(0)|_{\mathbb{E}}$ is separated from zero. Let us for every $k$ consider the integral in \eqref{EQ: ContinuityThBoundedIntegral} on segments $[l-1,l]$, where $l=-1,-2,\ldots$. Using the mean value theorem, we get a family of numbers $t^{(l)}_{k} \in [l-1,l]$ such that vectors $v_{k}(t^{(l)}_{k})$ are bounded in $\mathbb{E}$ uniformly in $k$ and the bound depends on $l$. Taking a subsequence, if necessary, we may assume that $t^{(l)}_{k} \to \overline{t}_{l}$ as $k \to +\infty$ for some $\overline{t}_{l} \in [l-1,l]$. For a fixed $l$ let us consider $\mathcal{C}:=\{ \vartheta^{t^{(l)}_{k}}(q_{k}) \ | \ k=1,2,\ldots \}$ and $\mathcal{B}:= \{ \Pi v_{k}(t^{(l)}_{k}) \ | \ k=1,2,\ldots \}$. Clearly, $\mathcal{C}$ is precompact and $\mathcal{B}$ is bounded. Then $v_{k}(t^{(l)}_{k})$ lies in the precompact (in $\mathbb{E}$) set $\Phi(\mathcal{C},\mathcal{B})$. Thus, there exist a converging subsequence. Using Cantor's diagonal argument, we may assume that $v_{k}(t^{(l)}_{k})$ converges in $\mathbb{E}$ to some $\overline{v}_{l} \in \mathbb{E}$ as $k \to +\infty$ for all $l=-1,-2,\ldots$. Let us define $\overline{v}_{l}(t):=\psi^{t-\overline{t}_{l}}(\vartheta^{\overline{t}_{l}}(q), \overline{v}_{l})$ for $t \geq \overline{t}_{l}$. From the continuity of the cocycle we have that $v_{k}(t) \to \overline{v}_{l}(t)$ as $k \to +\infty$ for $t > \overline{t}_{l}$. This implies that $\overline{v}_{l-1}(t)=\overline{v}_{l}(t)$ for all $l$ and all $t > t_{l}$. For any $t \in \mathbb{R}$ define $v^{*}(t):=\overline{v}_{l}(t)$, where $l$ such that $t > t_{l}$. Then $v^{*}(\cdot)$ is a complete trajectory at $q$ and $v_{k}(t) \to v^{*}(t)$ in $\mathbb{E}$ for all $t \in \mathbb{R}$. From \eqref{EQ: ContinuityThBoundedIntegral} it follows that $v^{*}(\cdot)$ is amenable. Since $\Pi v_{k}(0) = \zeta_{k} \to \Pi v^{*}(0)$ and $\zeta_{k} \to \zeta$, we must have $\Pi v^{*}(0) = \zeta$. But since $v(\cdot)$ and $v^{*}(\cdot)$ are both amenable at $q$ and $\Pi \colon \mathfrak{A}(q) \to \mathbb{E}^{-}$ is a homeomorphism, we must have $v(\cdot) \equiv v^{*}(\cdot)$ and, consequently, $v_{k}(0) \to v(0)$ in $\mathbb{E}$. This leads to a contradiction.
\end{proof}
\begin{remark}
	\label{REM: ApplicationsofContTh}
	Let us discuss the conditions of Theorem \ref{TH: PhiContinuityGen}. Assumption \textbf{(BA)} is natural in the following two cases. The first one is when $(\mathcal{Q},\vartheta)$ is a minimal almost periodic (in the sense of Bohr) flow. In this case this condition will be satisfied if there is a bounded in the future semitrajectory at some $q$, i.~e. when the limit dynamics is nontrivial. The second case is when the cocycle is linear (as in Section \ref{SEC: InvariantManifoldsDelay}). Here the zero trajectory at each $q$ satisfies the required property. The assumption of uniform compactness for $\Phi$ is also natural. It will be satisfied if we assume that assumptions \textbf{(UCOM)} and \textbf{(ULIP)} below hold. Thus, this condition is also natural.
\end{remark}

\begin{description}
	\item[\textbf{(UCOM)}] There exists $\tau_{ucom} > 0$ such that for any precompact set $\mathcal{C} \subset \mathcal{Q}$ and any bounded set $\mathcal{B} \subset \mathbb{E}$ the set $\psi^{\tau_{ucom}}(\mathcal{C},\mathcal{B})$ is precompact in $\mathbb{E}$.
\end{description}
\begin{description}
	\item[\textbf{(ULIP)}] There exists $t_{S} \geq 0$ such that for any $T>0$ there is a constant $L_{T}>0$ such that for all $v_{1},v_{2} \in \mathbb{E}$ and $q \in \mathcal{Q}$ we have
	\begin{equation}
	\|\psi^{t}(q,v_{1})-\psi^{t}(q,v_{2})\|_{\mathbb{E}} \leq L_{T}|v_{1}-v_{2} |_{\mathbb{H}} \text{ for } t \in [t_{S},t_{S}+T]
	\end{equation}
	and also
	\begin{equation}
	\|\psi^{t}(q,v_{1})-\psi^{t}(q,v_{2})\|_{\mathbb{E}} \leq C''_{T}\|v_{1}-v_{2} \|_{\mathbb{E}} \text{ for } t \in [0,T].
	\end{equation}
\end{description}

\begin{lemma}
	Under the hypotheses of Theorem \ref{TH: ReductionTheoremGener} suppose in addition that $\mathfrak{A}(q)$ is nonempty for all $q \in \mathcal{Q}$ and let \textbf{(BA)}, \textbf{(UCOM)} and \textbf{(ULIP)} be satisfied. Then $\Phi$ is uniformly compact.
\end{lemma}
\begin{proof}
Consider a precompact set $\mathcal{C} \subset \mathcal{Q}$ and a bounded set $\mathcal{B} \subset \mathbb{E}$. Consider any sequence $v_{0,k} \in \Phi(\mathcal{C}\times \mathcal{B})$, where $k=1,2,\ldots$, for which we shall find a converging subsequence. Then $v_{0,k} = \Phi(q_{k}, \zeta_{k})$ for some $q_{k} \in \mathcal{C}$, $\zeta_{k} \in \mathcal{B}$ and there are amenable trajectories $v^{*}_{k}(\cdot)$ at $q_{k}$ such that $v^{*}_{k}(0)=v_{0,k}$. We may assume that $q_{k}$ converges to some $q$ as $k \to +\infty$. Similarly to \eqref{EQ: ContIntegalDiffEstimate} from \textbf{(ULIP)} for some constant $\widetilde{M}>0$ we get
\begin{equation}
	\begin{split}
		\widetilde{M} &\geq \int_{-\infty}^{0} e^{2\nu s}|v^{*}_{k}(s)-w^{*}_{q_{k}}(s)|^{2}_{\mathbb{H}}ds \geq \int_{-\tau-1}^{-\tau}e^{2\nu s}|v^{*}_{k}(s)-w^{*}_{q_{k}}(s)|^{2}_{\mathbb{H}}ds =\\&= e^{2\nu s_{k}}|v^{*}_{k}(s_{k})-w^{*}_{q_{k}}(s_{k})|^{2}_{\mathbb{H}} \geq e^{-2\nu (\tau+1)} (C'_{\tau+1})^{-2} \| v^{*}_{k}(-\tau_{ucom})-w^{*}_{q_{k}}(-\tau_{ucom}) \|^{2}_{\mathbb{E}},
	\end{split}
\end{equation}
where $\tau = \operatorname{max}\{\tau_{S}, \tau_{ucom} \}$ and $s_{k} \in [-\tau-1,-\tau]$ are some numbers. From this and \textbf{(BA)} we get that $v^{*}_{k}(-\tau_{ucom})$ lies in some bounded set $\mathcal{B}'$ and, consequently, $v^{*}_{k}(0) \in \psi^{\tau_{ucom}}(\mathcal{C}',\mathcal{B}')$, where $\mathcal{C}'= \{ \vartheta^{-\tau_{ucom}}(q_{k}) \ | \ k=1,2,\ldots \}$. Then by \textbf{(UCOM)} one can find a converging subsequence of $v^{*}_{k}(0)=v_{0,k}$. The lemma is proved.
\end{proof}

\begin{lemma}
	\label{LEM: UniformLipschitProperty}
	Under the hypotheses of Theorem \ref{TH: ReductionTheoremGener} suppose in addition that \textbf{(ULIP)} is satisfied. Let $\mathfrak{A}(q)$ be nonempty for some $q \in \mathcal{Q}$. Then $\Phi_{q} \colon \mathbb{E}^{-} \to \mathbb{E}$ is Lipschitz and the Lipschitz constants are bounded uniformly in $q \in \mathcal{Q}$. In particular, $\mathfrak{A}(q)$ is a Lipschitz submanifold in $\mathbb{E}$.
\end{lemma}
\begin{proof}
	Similarly to \eqref{EQ: ContIntegalDiffEstimate} for any two amenable trajectories $v^{*}_{1}(\cdot)$ and $v^{*}_{2}(\cdot)$ at $q$ we have
	\begin{equation}
		\begin{split}
			\delta^{-1} \|P\| \cdot \| \Pi v^{*}_{1}(0) - \Pi v^{*}_{2}(0) \|^{2}_{\mathbb{E}} \geq \int_{-\infty}^{0} e^{2\nu s} | v^{*}_{1}(s)-v^{*}_{2}(s) |^{2}_{\mathbb{H}} ds \geq \\ \geq \int_{-\tau_{S}-1}^{-\tau_{S}}e^{2\nu s} | v^{*}_{1}(s)-v^{*}_{2}(s) |^{2}_{\mathbb{H}} = e^{2\nu s_{0}} | v^{*}_{1}(s_{0})-v^{*}_{2}(s_{0}) |^{2}_{\mathbb{H}}
		\end{split}	
	\end{equation}
	for some $s_{0} \in [-\tau_{S}-1,-\tau_{S}]$. From \textbf{(ULIP)} we have
	\begin{equation}
		\| v^{*}_{1}(0)-v^{*}_{2}(0) \|_{\mathbb{E}} \leq C'_{\tau_{S}+1} | v^{*}_{1}(s_{0})-v^{*}_{2}(s_{0}) |_{\mathbb{H}}.
	\end{equation}
	Now if $\Pi v^{*}_{i}(0) = \zeta_{i}$ for $i=1,2$, we get
	\begin{equation}
		\label{EQ: PhiQLipschitzConstant}
		\begin{split}
			\| \Phi_{q}(\zeta_{1})-\Phi_{q}(\zeta_{2}) \|_{\mathbb{E}} = \| v^{*}_{1}(0)-v^{*}_{2}(0) \|_{\mathbb{E}} \leq C'_{\tau_{S}+1} e^{\nu(\tau_{S}+1)} \sqrt{\delta^{-1} \| P \|} \cdot \| \zeta_{1}-\zeta_{2} \|_{\mathbb{E}}.
		\end{split}	
	\end{equation}
	Thus, we may put $C_{\Phi} := C'_{\tau_{S}+1} e^{\nu(\tau_{S}+1)} \sqrt{\delta^{-1} \|P\|}$ and this shows the required statement.
\end{proof}

\bibliographystyle{amsplain}

\end{document}